\documentclass[reqno, 11pt]{amsart}
\usepackage{bbm,calc,amsfonts,amsthm,amscd,epsfig,psfrag,amsmath,amssymb,enumerate,graphicx}

\theoremstyle{plain}
\newtheorem{theorem}{Theorem}[section]
\newtheorem{proposition}[theorem]{Proposition}
\newtheorem{corollary}[theorem]{Corollary}
\newtheorem{lemma}[theorem]{Lemma}

\theoremstyle{remark}
\newtheorem{remark}[theorem]{Remark}
\newtheorem{definition}[theorem]{Definition}

\numberwithin{equation}{section}

\newcommand{\cdummy}{\cdot}
\newcommand{\mathd}{\mathrm{d}}
\newcommand{\nin}{\not\in}
\newcommand{\nocomma}{}
\newcommand{\tmem}[1]{{\em #1\/}}
\newcommand{\tmmathbf}[1]{\ensuremath{\boldsymbol{#1}}}
\newcommand{\tmop}[1]{\ensuremath{\operatorname{#1}}}
\newcommand{\tmstrong}[1]{\textbf{#1}}

\newcommand{\tmtextit}[1]{{\itshape{#1}}}

\newcommand{\mnote}[1]{\marginpar{\scriptsize\it\begin{minipage}[t]{1.4in} \raggedright #1\end{minipage}}}
\def\1{\ifmmode {1\hskip -3pt \rm{I}} \else {\hbox {$1\hskip -3pt
      \rm{I}$}}\fi}

\begin{document}

\title{Glauber dynamics for the quantum Ising model in a transverse
  field on a regular tree}

\author[F.~Martinelli]{Fabio Martinelli} 
\address{F. Martinelli, Dipartimento di Matematica,
  Universit\`a Roma Tre, Largo S.\ Murialdo 1, 00146 Roma, Italy.
  e--mail: {\tt martin@mat.uniroma3.it}} 
\author[M.~Wouts]{Marc Wouts}
\address{M. Wouts, D\'epartment de Math\'ematiques, Universit\'e Paris 13, 99 Avenue J.B. Cl\'ement, F-93430
  Villetaneuse, France. e--mail: {\tt wouts@math.univ-paris13.fr}}
\thanks{This work was supported by the European Research
Council throughout the ``Advanced Grant'' PTRELSS 228032.}
\maketitle

\begin{abstract}
  Motivated by a recent use of Glauber dynamics for Monte-Carlo simulations of
  path integral representation of quantum spin models {\cite{KRSZ08}}, we
  analyse a natural Glauber dynamics for the quantum Ising model with a
  transverse field on a finite graph $G$. We establish strict monotonicity
  properties of the equilibrium distribution and we extend (and improve) the
  censoring inequality of Peres and Winkler to the quantum setting. Then we
  consider the case when $G$ is a regular $b$-ary tree and prove the same fast
  mixing results established in {\cite{MSW04}} for the classical Ising model.
  Our main tool is an inductive relation between conditional marginals (known
  as the ``cavity equation'') together with sharp bounds on the operator norm
  of the derivative at the stable fixed point. It is here that the main
  difference between the quantum and the classical case appear, as the cavity
  equation is formulated here in an infinite dimensional vector space, whereas
  in the classical case marginals belong to a one-dimensional space.\\
  
  {\noindent}\tmtextit{{\noindent}AMS 2000 subject classification:} Primary:
  60K35; Secondary: 82C10.\\
  {\noindent}\tmtextit{Keywords:} Quantum Ising model, Glauber dynamics, stochastic monotonicity, cavity equation, mixing time, spectral gap.
\end{abstract}

\section{Introduction}
In the last years there has been a somewhat intense research in probabilistic representation of quantum spin systems with very interesting progresses exploiting 
stochastic geometry methods. We refer the interested reader to  \cite{Io09} and references therein. More recently, physicists started to consider quantum Heat Bath algorithms applied to the imaginary time path integral representation of e.g.~the quantum ising model in a transverse field (see \cite{KRSZ08} and references therein). 
From a mathematical point of view that accounts to introduce a so-called Glauber dynamics for the dynamical variables appearing in the path integral representation of the model, a subject that, in the context of \emph{classical} spin models, has been thoroughly investigated in the mathematical literature in the past fifteen years because of its many different facets and connections with different areas in mathematics, physics and theoretical computer science. In the quantum setting very little is known rigorously and we are only aware of a recent contribution  \cite{CDP} where, for a particular choice of the Glauber chain, exponentially fast approach to the imaginary time equilibrium measure is proved via a kind of Bochner-Bakry-Emery method in a high temperature regime. 

Here we investigate the simple heat bath dynamics for the imaginary time path integral representation of the quantum Ising model with a transverse field on a finite connected graph $G$. We first focus on properties of the single site measure and establish key \emph{monotonicity} properties as a function of the longitudinal magnetic field and of  the boundary conditions. A remarkable consequence of monotonicity is the validity of the so-called Peres-Winkler censoring inequality \cite{PeresUBC} which we re-establish and extend to the quantum case. Then, motivated by the analysis of the same model made by physicists on the Bethe lattice \cite{KRSZ08}, we specialize to the case when the graph $G$ is a regular $b$-ary tree and we succeed in extending to the quantum case much of the analysis made in \cite{MSW04} for the classical setting. More precisely we prove that, for any boundary condition $\tau$, there exists exponents $\kappa=\kappa(\tau)$, $\gamma<1$ such that, if $\kappa \gamma b<1$,  then the Glauber chain on a finite subtree of depth $\ell$ with boundary condtion $\tau$ at the leaves, has a spectral gap uniformly positive in $\ell$. In \cite{MSW04} the same uniformity was proved for the logarithmic Sobolev constant and, as a consequence, the mixing time cound not grow faster than linear in $\ell$. In the quantum setting the  logarithmic Sobolev constant  is easily seen to be infinite. Nevertheless, using our extension of censoring and always assuming $\kappa \gamma b < 1$, we prove that the mixing time grows linearly in $\ell$. Finally, in complete analogy with the fast mixing result inside the pure plus phase for the classical Ising model, we show that, if the boundary condition is the homogeneous \emph{plus} configuration then the key exponent $\kappa$ satisfies $\kappa \leqslant 1/b$ for all values of the thermodynamic parameters. Therefore, in this case, $\kappa \gamma b<1$. The same holds  for \emph{all} boundary conditions when the thermodynamic parameters are in the uniqueness regime.

All the main technical efforts in the paper are concentrated on
bounding the exponent $\kappa$ and it is at this point that the
analysis of the quantum case becomes much more involved than its
classical analog. Computing the exponent $\kappa$ requires in fact a
detailed study of the quantum cavity equation for the one site
conditional marginal distributions and, in particular, of the
derivative operator at the fixed point corresponding to the ``plus'' phase. In the classical case such marginals are parametrized by just one number and therefore the corresponding cavity equation is just a recursive non-linear equation in $\mathbf R$. In the quantum case instead the cavity equation becomes a recursive functional equation in a  infinite dimensional vector space.
  
{\tableofcontents}

\section{The quantum Ising model}

\subsection{The matrix formulation}

We consider a finite graph $G = (V, E)$ with vertex set $V$ and edge set $E$.
A classical spin configuration on $V$ is $\sigma = (\sigma_i)_{i \in V} \in
\{\pm 1\}^V$. Now we define $\mathcal{H}$, a $2^{|V|}$-dimensional Hilbert
space, by giving the orthonormal basis $\left. \left( | \sigma \right\rangle
\right)_{\sigma \in \{\pm 1\}^V}$. We also denote $\langle \sigma |$ the
transpose of any basis vector $| \sigma \rangle$. Given $i \in V$, we define
two linear operators on $\mathcal{H}$ by their action on the orthonormal
basis:
\begin{eqnarray*}
  \sigma_i^z | \sigma \rangle & = & \sigma_i | \sigma \rangle\\
  \sigma_i^x | \sigma \rangle & = & | \sigma^i \rangle
\end{eqnarray*}
where $\sigma^i$ is the same spin configuration as $\sigma$, except for the
spin at $i$ which is flipped. Given a set of parameters $\beta, h, \lambda$
($\beta>0$ is the inverse temperature, $h\in{\mathbf R}$ the longitudinal and $\lambda \geqslant 0$ the
transverse field) we define the Hamiltonian operator
\begin{eqnarray*}
  H & = & - \sum_{e =\{i, j\} \in E} \sigma_i^z \sigma_j^z - h \sum_{i \in V}
  \sigma_i^z - \lambda \sum_{i \in V} \sigma_i^x
\end{eqnarray*}
and define the average of an observable $O$ (a linear operator on
$\mathcal{H}$) at inverse temperature $\beta$ as
\begin{eqnarray*}
  \langle O \rangle_{\beta, h, \lambda} & = & \frac{\tmop{Tr} \left( Oe^{-
  \beta H} \right)}{\tmop{Tr} \left( e^{- \beta H} \right)} .
\end{eqnarray*}

\subsection{The path-integral representation of the quantum Ising model}

We introduce first a reference measure $\varphi_{I, \lambda}$ that corresponds
to a single free spin on some interval $I \subset {\mathbf R}$. A single spin
is a cadlag function $\tmmathbf{\sigma}: I \rightarrow \left\{ \pm 1
\right\}$. We call $\Sigma_I$ the set of single spins, and endow it with the
usual Skorohod topology. Note that two spin configurations are close when the
have the same initial value, same number of flips and close time location for
the flips. This set is a Polish space as the set of spin configurations with
jumps at fractional times is countable and dense in $\Sigma$. The
corresponding $\sigma$-algebra is generated by the events $\left\{
\tmmathbf{\sigma} \left( t \right) = + \right\}$, for $t \in I$. We will also
consider the $L^1$-distance
\begin{eqnarray*}
  \left\| \tmmathbf{\sigma}-\tmmathbf{\sigma}' \right\|_1 & = & \int_I \left|
  \tmmathbf{\sigma} \left( t \right) -\tmmathbf{\sigma}' \left( t \right)
  \right| \mathd t .
\end{eqnarray*}
between two spin configurations. Now we define the distribution $\varphi_{I,
\lambda}$ as follows: we consider a Poisson point process with intensity
$\lambda$ on $I$, and define $\tmmathbf{\sigma}$, with equal probability, as
one of the two configurations that flip exactly at the points of the Poisson
process.

Now we introduce spin interactions and external fields. Given a bounded
interval $I$ and two bounded fonctions $\tmmathbf{\sigma}, \tmmathbf{\tau}: I
\rightarrow {\mathbf R}$ defined Lebesgue a.s., we call
\begin{eqnarray}
  \tmmathbf{\sigma} \underset{I}{\bullet} \tmmathbf{\tau} & = & \int_I
  \tmmathbf{\sigma}(t)\tmmathbf{\tau}(t) \mathd t .  \label{eq:dot}
\end{eqnarray}
When $I$ is clear from the context (usually $I = [0, \beta]$), we drop it from
the notation. Note that the product defined by (\ref{eq:dot}) is additive in
the sense that, for $I_1, I_2$ disjoint,
\begin{eqnarray*}
  \tmmathbf{\sigma} \underset{I_1 \cup I_2}{\bullet} \tmmathbf{\tau} & = &
  \tmmathbf{\sigma} \underset{I_1}{\bullet} \tmmathbf{\tau}+\tmmathbf{\sigma}
  \underset{I_2}{\bullet} \tmmathbf{\tau}.
\end{eqnarray*}
A spin configuration is $\tmmathbf{\sigma}: V \times I \rightarrow \left\{ \pm
1 \right\}$ such that $\tmmathbf{\sigma}_i \in \Sigma_I$, for all $i \in V$.
Consider some $\beta > 0$ (inverse temperature), $\tmmathbf{h}: V \times
\left[ 0, \beta \right] \rightarrow {\mathbf R}$ integrable (the longitudinal
field), $\lambda \geqslant 0$ (the transverse field). We define two probability
measures on the set of spins configurations:
\begin{eqnarray*}
  \mu_{G ; \beta, \tmmathbf{h}, \lambda} (\mathd \tmmathbf{\sigma}) & = &
  \frac{1}{Z_{G ; \beta, \tmmathbf{h}, \lambda}} \exp \left( \sum_{i \sim j}
  \tmmathbf{\sigma}_i \bullet \tmmathbf{\sigma}_j + \sum_i \tmmathbf{h}_i
  \bullet \tmmathbf{\sigma}_i \right) \prod_{i \in V} \varphi_{\left[ 0, \beta
  \right], \lambda} (\mathd \tmmathbf{\sigma}_i)\\
  \mu^{\tmop{per}}_{G ; \beta, \tmmathbf{h}, \lambda} (\mathd
  \tmmathbf{\sigma}) & = & \frac{\tmmathbf{1}_{\{\tmmathbf{\sigma}(0)
  =\tmmathbf{\sigma}(\beta)\}}}{Z^{\tmop{per}}_{G ; \beta, \tmmathbf{h},
  \lambda}} \exp \left( \sum_{i \sim j} \tmmathbf{\sigma}_i \bullet
  \tmmathbf{\sigma}_j + \sum_i \tmmathbf{h}_i \bullet \tmmathbf{\sigma}_i
  \right) \prod_{i \in V} \varphi_{\left[ 0, \beta \right], \lambda} (\mathd
  \tmmathbf{\sigma}_i)
\end{eqnarray*}
where $Z_{G ; \beta, \tmmathbf{h}, \lambda}$ and $Z^{\tmop{per}}_{G ; \beta,
\tmmathbf{h}, \lambda}$ are the appropriate normalization constant that turn
$\mu_{G ; \beta, \tmmathbf{h}, \lambda}$ and $\mu_{G ; \beta, \tmmathbf{h},
\lambda}^{\tmop{per}}$ into probability measures on the set of spin
configurations. They exists as the argument of the exponential is trivially
bounded by $\beta \left| E \right| + \beta \left\| \tmmathbf{h}
\right\|_{1} \left| V \right|$. The boundary condition $\tmmathbf{\sigma}
\left( 0 \right) =\tmmathbf{\sigma} \left( \beta \right)$ is called
{\tmem{periodic imaginary time boundary condition}}. Note that we could also
allow $\lambda : V \times I \rightarrow {\mathbf R}^+$.

The above measure are often approximated, in the physics literature, by a
discrete Ising model with vertex set $V \times \left\{ 0, \ldots, N \right\}$
and edge set $E \times \left\{ 0, \ldots, N \right\} \cup V \times \left\{
\left( 0, 1 \right), \ldots, \left( N - 1, N \right), \left( N, 0 \right)
\right\}$. This is usually referred to as a consequence of the Suzuki-Trotter
formula. The strength of interaction is different for transversal edges than
for longitudinal edges. Although we do not use this approximation (which
corresponds to enabling the discontinuities only at times $\beta k / N$), it
explains that the above two measures share most of the properties of discrete
Ising models.

\begin{proposition}
  {\tmstrong{[DLR Equation]}}. Let $I = \left[ \beta_1, \beta_2 \right]
  \subset \left[ 0, \beta \right]$ be an interval and $A \subset V$. Let
  $\tmmathbf{\sigma} \sim \mu_{G ; \beta, \tmmathbf{h}, \lambda}$.
  Conditionally on $\tmmathbf{\sigma}_{\left| \left( A \times I \right)^c
  \right.}$, the distribution of $\tmmathbf{\sigma}_{\left| A \times I
  \right.}$ is the quantum Ising model $\mu_{\left( A, E \right) ; I,
  \tmmathbf{h}', \lambda} \left( . | \tmmathbf{\sigma}_A \left( \beta_1
  \right), \tmmathbf{\sigma}_A \left( \beta_2 \right) \right)$ with external
  field $\tmmathbf{h}' =\tmmathbf{h}+ \sum_{i \in A, j \nin A}
  \tmmathbf{\sigma}_j \delta_i$.
\end{proposition}

\begin{proof}
  This is a consequence of the independence of the restrictions of the Poisson
  point process to different regions together with the additivity of the
  product (\ref{eq:dot}).
\end{proof}

\begin{proposition}
  Let $\beta > 0$, $h \in {\mathbf R}$ and $\lambda \geqslant 0$. Then, for any
  observable $O$,
  \begin{eqnarray}
    \langle O \rangle_{\beta, h, \lambda} & = & \frac{\mu_{V ; \beta, h,
    \lambda} (\langle \tmmathbf{\sigma}(\beta) |O|\tmmathbf{\sigma}(0)
    \rangle)}{\mu_{V ; \beta, h, \lambda} (\tmmathbf{\sigma}(0)
    =\tmmathbf{\sigma}(\beta))}  \label{avgOnd}
  \end{eqnarray}
  while, for any diagonal observable $O$,
  \begin{eqnarray}
    \langle O \rangle_{\beta, h, \lambda} & = & \mu^{\tmop{per}}_{V ; \beta,
    h, \lambda} (\langle \tmmathbf{\sigma}(\beta) |O|\tmmathbf{\sigma}(0)
    \rangle) . 
  \end{eqnarray}
\end{proposition}

\begin{proof}
  We define a matrix $W_{\beta, h, \lambda}$ by the prescription of its
  coordinates
  \begin{eqnarray*}
    \langle \sigma |W_{\beta, h, \lambda} | \sigma' \rangle & = & 2^{|V|} \int
    \exp \left( \sum_{i \sim j} \tmmathbf{\sigma}_i \bullet
    \tmmathbf{\sigma}_j + h \sum_i \tmmathbf{1} \bullet \tmmathbf{\sigma}_i
    \right) \tmmathbf{1}_{\{\tmmathbf{\sigma}(0) = \sigma,
    \tmmathbf{\sigma}(\beta) = \sigma' \}} \prod_{i \in V} \mathd \varphi_{[0,
    \beta], \lambda} (\tmmathbf{\sigma}_i) .
  \end{eqnarray*}
  Once we prove that
  \begin{eqnarray}
    W_{\beta, h, \lambda} & = & e^{- \beta \lambda |V|} \times \exp \left( -
    \beta H \right)  \label{Wexp}
  \end{eqnarray}
  the claim follows from the definition of the average of an observable:
  \begin{gather*}
    \langle O \rangle_{\beta, h, \lambda}  =  \frac{\tmop{Tr} \left( Oe^{-
    \beta H} \right)}{\tmop{Tr} \left( e^{- \beta H} \right)}
     =  \frac{\sum_{\sigma} \langle \sigma |O e^{- \beta H} | \sigma
    \rangle}{\sum_{\sigma} \langle \sigma | e^{- \beta H} | \sigma \rangle}\\
     =  \frac{\sum_{\sigma, \sigma'} \langle \sigma |O| \sigma' \rangle
    \langle \sigma' | e^{- \beta H} | \sigma \rangle}{\sum_{\sigma} \langle
    \sigma | e^{- \beta H} | \sigma \rangle}
  \end{gather*}
  as
  \begin{gather*}
    \langle \sigma' | e^{- \beta H} | \sigma \rangle =  e^{\beta \lambda
    |V|} \langle \sigma' |W_{\beta, h, \lambda} | \sigma \rangle\\
     =  e^{\beta \lambda |V|} 2^{|V|} Z_{V ; \beta, h, \lambda}\ \mu_{V ;
    \beta, h, \lambda} (\tmmathbf{\sigma}(0) = \sigma',
    \tmmathbf{\sigma}(\beta) = \sigma) . 
  \end{gather*}
  So we focus now on the proof of (\ref{Wexp}). According to the additivity of
  the product (\ref{eq:dot}) and to the independence of the restriction of the
  Poisson process to $\left[ 0, \beta_1 \right]$ and $\left[ \beta_1, \beta_2
  \right]$, we have
  \begin{eqnarray}
    W_{\beta_1 + \beta_2, h, \lambda} & = & W_{\beta_1, h, \lambda} \times
    W_{\beta_2, h, \lambda} \text{,} \forall \beta_1, \beta_2 \geqslant 0. 
    \label{W:mult}
  \end{eqnarray}
  Now we compute the asymptotic of $\langle \sigma |W_{\beta, h, \lambda} |
  \sigma' \rangle$ when $\beta \rightarrow 0$. The probability that the
  Poisson point process has two points or more on $V \times [0, \beta]$ is $O
  (\beta^2)$, therefore when $\sigma$ and $\sigma'$ differ at two points or
  more, $\langle \sigma |W_{\beta, h, \lambda} | \sigma' \rangle = O
  (\beta^2)$. When $\sigma' = \sigma^i$, the probability that
  $\tmmathbf{\sigma}_i$ is discontinuous is $\beta \lambda + O (\beta^2)$,
  therefore $\langle \sigma |W_{\beta, h, \lambda} | \sigma' \rangle = \beta
  \lambda + O (\beta^2) = \beta \langle \sigma |H| \sigma' \rangle + O
  (\beta^2)$. Finally, when $\sigma' = \sigma$, up to $O (\beta^2)$ only the
  configurations with no flips (probability $\exp (- \beta \lambda |V|)$)
  contribute to $\langle \sigma |W_{\beta, h, \lambda} | \sigma' \rangle$,
  therefore in that case
  \begin{eqnarray*}
    \langle \sigma |W_{\beta, h, \lambda} | \sigma' \rangle & = & \exp \left(
    \beta \left( \sum_{i \sim j} \sigma_i \sigma_j + h \sum_i \sigma_i \right)
    - \beta \lambda |V| \right) + O (\beta^2)\\
    & = & 1 - \beta \langle \sigma |H| \sigma' \rangle - \beta \lambda |V| +
    O (\beta^2) .
  \end{eqnarray*}
  In other words we have shown that
  \begin{eqnarray}
    W_{\beta, h, \lambda} & = & (1 - \beta \lambda |V|) I - \beta H + O
    (\beta^2)  \label{W:betasmall}
  \end{eqnarray}
  as $\beta \rightarrow 0$, where $I$ is the matrix that corresponds to the
  identity of $\mathcal{H}$. Combining (\ref{W:mult}) and (\ref{W:betasmall})
  we obtain, for any $\beta \geqslant 0$ and any $\sigma'$,
  \begin{eqnarray*}
    \frac{\mathd}{\mathd \beta} W_{\beta, h, \lambda} | \sigma' \rangle & = &
    \left[ \frac{\mathd}{\mathd \beta'} W_{\beta', h, \lambda} \right]_{\beta'
    = 0} \times W_{\beta, h, \lambda} | \sigma' \rangle\\
    & = & \left( - \lambda |V|I - H \right) \times W_{\beta, h, \lambda} |
    \sigma' \rangle .
  \end{eqnarray*}
  This proves (\ref{Wexp}) as $\exp (- \beta \lambda |V| I - \beta H)$ is the
  unique solution to this differential equation.
\end{proof}

\section{The single site spin measure}

In all this section we assume that $V = \{x\}$ is a single site. Most of the
time $\beta, \lambda$ and sometimes $\tmmathbf{h}$ will be  from the context
and they will not appear in the notation. We state several important facts about
this single site measure. In what follows $\|\cdot\|_{\tmop{TV}}$ denotes the usual
variation distance (see also (\ref{TV}).

\begin{proposition}
  \label{prop:gamma}Let $\beta, M < \infty$. There exists $\gamma < 1$ and
  $\Gamma = 1 / \log (3)$ such that, for any $\tmmathbf{h}, \tmmathbf{h}',
  \lambda$ such that $\|\tmmathbf{h}\|_1, \|\tmmathbf{h}' \|_1, \lambda < M$,
  for any $\pi \in \{\emptyset, \tmop{per}\}$, the following inequalities
  holds
  \begin{eqnarray}
    \left\| \mu^{\pi}_{\tmmathbf{h}'} - \mu^{\pi}_{\tmmathbf{h}}
    \right\|_{\tmop{TV}} & \leqslant & \gamma  \label{mu:gamma}\\
    \left\| \mu^{\pi}_{\tmmathbf{h}'} - \mu^{\pi}_{\tmmathbf{h}}
    \right\|_{\tmop{TV}} & \leqslant & \Gamma \|\tmmathbf{h}'
    -\tmmathbf{h}\|_1 .  \label{mu:Gamma}
  \end{eqnarray}
\end{proposition}
\begin{proof}
  Inequality (\ref{mu:gamma}) is an immediate consequence of the fact that, under our assumptions,
  $\mu_{\tmmathbf{h}} (\{ + \})$ is uniformly bounded from below. For (\ref{mu:Gamma}) we introduce
  $\varphi^{\pi}$, the single spin free measure conditioned on the imaginary
  time boundary condition $\pi$ and compute
  \begin{eqnarray*}
    \left\| \mu^{\pi}_{\tmmathbf{h}'} - \mu^{\pi}_{\tmmathbf{h}}
    \right\|_{\tmop{TV}} & = & \frac{1}{2} \int \left| \frac{e^{\tmmathbf{h}'
    \bullet \tmmathbf{\sigma}}}{\int e^{\tmmathbf{h}' \bullet
    \tmmathbf{\sigma}'} \mathd \varphi^{\pi} (\tmmathbf{\sigma}')} -
    \frac{e^{\tmmathbf{h} \bullet \tmmathbf{\sigma}}}{\int e^{\tmmathbf{h}
    \bullet \tmmathbf{\sigma}'} \mathd \varphi^{\pi} (\tmmathbf{\sigma}')}
    \right| \mathd \varphi^{\pi} (\tmmathbf{\sigma})\\
    & \leqslant & \frac{1}{2}  \frac{\int \left| e^{\tmmathbf{h}' \bullet
    \tmmathbf{\sigma}+ h \bullet \tmmathbf{\sigma}'} - e^{\tmmathbf{h} \bullet
    \tmmathbf{\sigma}+\tmmathbf{h}' \bullet \tmmathbf{\sigma}'} \right| \mathd
    \varphi^{\pi} (\tmmathbf{\sigma}) \mathd \varphi^{\pi}
    (\tmmathbf{\sigma}')}{\int e^{\tmmathbf{h} \bullet
    \tmmathbf{\sigma}+\tmmathbf{h}' \bullet \tmmathbf{\sigma}'} \mathd
    \varphi^{\pi} (\tmmathbf{\sigma}) \mathd \varphi^{\pi}
    (\tmmathbf{\sigma}')}\\
    & = & \frac{1}{2}  \frac{\int e^{\tmmathbf{h} \bullet
    \tmmathbf{\sigma}+\tmmathbf{h}' \bullet \tmmathbf{\sigma}'} \left|
    e^{(\tmmathbf{h}' -\tmmathbf{h}) \bullet
    (\tmmathbf{\sigma}-\tmmathbf{\sigma}')} - 1 \right| \mathd \varphi^{\pi}
    (\tmmathbf{\sigma}) \mathd \varphi^{\pi} (\tmmathbf{\sigma}')}{\int
    e^{\tmmathbf{h} \bullet \tmmathbf{\sigma}+\tmmathbf{h}' \bullet
    \tmmathbf{\sigma}'} \mathd \varphi^{\pi} (\tmmathbf{\sigma}) \mathd
    \varphi^{\pi} (\tmmathbf{\sigma}')}\\
    & \leqslant & \frac{\exp (\|\tmmathbf{h}' -\tmmathbf{h}\|_1) - 1}{2} .
  \end{eqnarray*}
  The claim follows from the inequality $\left( e^x - 1 \right) / 2 \leqslant
  \Gamma x$, for every $x > 0$ such that $\Gamma x \leqslant 1$.
\end{proof}

The second property is {\tmem{strict monotonicity}}, a crucial fact for our
analysis of the dynamics. Let $f : \Sigma \rightarrow {\mathbf R}$ be a
function of a single spin configuration. We say that $f$ is increasing when
$\tmmathbf{\sigma} \left( t \right) \leqslant \tmmathbf{\sigma}' \left( t
\right)$ for almost every $t \in \left[ 0, \beta \right]$ implies $f \left(
\tmmathbf{\sigma} \right) \leqslant f \left( \tmmathbf{\sigma}' \right)$. When
$\nu, \nu'$ are two probability measures on $\Sigma$ such that, for all
increasing function $f$, $\nu \left( f \right) \leqslant \nu' \left( f
\right)$, we say that $\nu$ is stochastically smaller than $\nu'$ and we will
write $\nu \prec \nu'$.

\begin{theorem}
  \label{thm:mono}Let $\beta, M < \infty$. There exists $c > 0$ such that, for
  any $\tmmathbf{h}, \tmmathbf{h}', \lambda$ such that $\|\tmmathbf{h}\|_1,
  \|\tmmathbf{h}' \|_1, \lambda < M$, for any $\pi \in \{\emptyset,
  \tmop{per}\}$, if $\tmmathbf{h}' \geqslant \tmmathbf{h}$ point-wise a.s.~then
  $\mu^{\pi}_{\tmmathbf{h}'} \succ \mu^{\pi}_{\tmmathbf{h}}$. Moreover, for
  any $f$ increasing,
  \begin{eqnarray}
    \mu^{\pi}_{\tmmathbf{h}'} \left( f \right) - \mu^{\pi}_{\tmmathbf{h}}
    \left( f \right) & \geqslant & c (f ( +) - f (-)) (\tmmathbf{h}'
    -\tmmathbf{h}) \bullet \tmmathbf{1}.  \label{mu:stpos}
  \end{eqnarray}
\end{theorem}

A consequence of monotonicity is the FKG inequality:

\begin{corollary}
  \label{cor:FKG}Let $\beta, \lambda < \infty$ and $\tmmathbf{h}$ with
  $\|\tmmathbf{h}\|_{\infty} < \infty$. Then, for any $\pi \in \{\emptyset,
  \tmop{per}\}$,any $f, g$ increasing,
  \begin{eqnarray}
    \mu^{\pi} \left( fg \right) & \geqslant & \mu^{\pi} \left( f \right)
    \mu^{\pi} \left( g \right) 
  \end{eqnarray}
\end{corollary}

The main idea in the proof of Theorem \ref{thm:mono} is an explicit coupling
of two single site measures with the same parameters but different imaginary
time boundary condition.

\begin{lemma}
  \label{lem:coupling}Fix $\beta, \lambda > 0$ and $\tmmathbf{h} \in L^1
  \left( \left[ 0, \beta \right] \right)$. Consider two independent spins
  variables $\tmmathbf{\sigma}^+_1 \sim \mu (.|\tmmathbf{\sigma}(\beta) = +)$
  and $\tmmathbf{\sigma}^-_1 \sim \mu (.|\tmmathbf{\sigma}(\beta) = -)$. Call
  $T^+$ (respectively $T^-$) the last flipping time (0 if none) of
  $\tmmathbf{\sigma}^+_1$ (respectively of $\tmmathbf{\sigma}_1^-$), and $T =
  \max \left( T^+, T^- \right)$. Consider the joint distribution $\Psi$ on
  $\left( \tmmathbf{\sigma}^+, \tmmathbf{\sigma}^- \right)$ as follows:
  \begin{enumerate}
    \item Let $\tmmathbf{\sigma}_{\left( T, \beta \right]}^+ = +$ and
    $\tmmathbf{\sigma}_{\left( T, \beta \right]}^- = -$.
    
    \item Take $\tmmathbf{\sigma}^+_{\left[ 0, T \right)}
    =\tmmathbf{\sigma}^-_{\left[ 0, T \right)} \sim \mu
    (.|\tmmathbf{\sigma}(T) = \varepsilon)$ where $\varepsilon = +$ if $T^+ <
    T^-$, $\varepsilon = -$ otherwise.
  \end{enumerate}
  Then $\Psi$ is a monotone coupling of $\mu (.|\tmmathbf{\sigma}(\beta) = +)$
  and $\mu (.|\tmmathbf{\sigma}(\beta) = -)$. It satisfies
  \begin{eqnarray}
    \Psi (\tmmathbf{\sigma}^+ = +, \tmmathbf{\sigma}^- = -) & = & \mu
    (\tmmathbf{\sigma}= + | \sigma (\beta) = +) \times \mu (\tmmathbf{\sigma}=
    - | \sigma (\beta) = -) .  \label{Psi:pm}
  \end{eqnarray}
\end{lemma}

\begin{proof}[Proof of Lemma \ref{lem:coupling}] $\Psi$ is monotone by construction
  ($\tmmathbf{\sigma}^+ \geqslant \tmmathbf{\sigma}^-$ point-wise a.s.) while
  (\ref{Psi:pm}) is a consequence of $\tmmathbf{\sigma}^{\pm} = \pm
  \Leftrightarrow T^+ = T^- = 0$. In order to prove that $\Psi$ is a coupling
  we need to check that it has the correct marginals. Let $t \in \left( 0,
  \beta \right)$. Conditioning on $T^+ < t$ is the same as conditioning the
  Poisson point process on having no point in $\left( t, \beta \right)$ and
  therefore $\tmmathbf{\sigma} \sim \mu (.|\tmmathbf{\sigma}(\beta) = +, T^+ <
  t)$, restricted to $\left( 0, t \right)$, has distribution $\mu_{t,
  \tmmathbf{h}, \lambda} (.|\tmmathbf{\sigma}(t) = +)$. Conditioning on $T^+ =
  t$ is more delicate as this is a zero probability event, still it has a very
  precise meaning in terms of the Poisson point process, since it requires
  that there is a point at $t$ and not point in $\left( t, \beta \right)$. It
  is well known that the conditional distribution of the points in $\left( 0,
  t \right)$ is an independent Poisson point process with the same intensity.
  Therefore, $\tmmathbf{\sigma} \sim \mu (.|\tmmathbf{\sigma}(\beta) = +, T^+
  = t)$, restricted to $\left( 0, t \right)$, has distribution $\mu_{t,
  \tmmathbf{h}, \lambda} (.|\tmmathbf{\sigma}(t) = -)$. The same holds if we
  replace $+$ by $-$ and vice-versa. Since $T^+$ and $T^-$ are independent,
  this proves that $\Psi$ is indeed a coupling of $\mu
  (.|\tmmathbf{\sigma}(\beta) = +)$ and $\mu (.|\tmmathbf{\sigma}(\beta) =
  -)$.
\end{proof}

\begin{proof}[Proof of Theorem \ref{thm:mono}] We begin with the case without periodic imaginary
  time boundary conditions, that is $\pi = \emptyset$. It is enough to
  quantify the influence of a unitary increase of the field, so we consider
  $\Delta \tmmathbf{h}: [0, \beta] \rightarrow {\mathbf R}^+$. We remark that
  \begin{eqnarray}
    \left[ \frac{\mathd}{\mathd s} \mu_{\tmmathbf{h}+ s \Delta \tmmathbf{h}}
    (f) \right]_{s = 0} & = & \int_0^{\beta} \mathd t \Delta \tmmathbf{h}(t)
    \tmop{Cov}_{\mu_{\tmmathbf{h}}} (\tmmathbf{\sigma}(t), f
    (\tmmathbf{\sigma}))  \label{eps}
  \end{eqnarray}
  and then
  \begin{eqnarray*}
    \tmop{Cov}_{\mu_{\tmmathbf{h}}} (\tmmathbf{\sigma}(t), f
    (\tmmathbf{\sigma})) & = & \frac{1}{2} \int \mu_{\tmmathbf{h}} \left(
    \mathd \tmmathbf{\sigma} \right) \mu_{\tmmathbf{h}} \left( \mathd
    \tmmathbf{\sigma}' \right) \left( \tmmathbf{\sigma}(t) -\tmmathbf{\sigma}'
    (t) \right) \left( f \left( \tmmathbf{\sigma} \right) - f \left(
    \tmmathbf{\sigma}' \right) \right)\\
    & = & 2 \int_{\tmmathbf{\sigma} \left( t \right) = +, \tmmathbf{\sigma}'
    \left( t \right) = -} \mu_{\tmmathbf{h}} \left( \mathd \tmmathbf{\sigma}
    \right) \mu_{\tmmathbf{h}} \left( \mathd \tmmathbf{\sigma}' \right) \left(
    f \left( \tmmathbf{\sigma} \right) - f \left( \tmmathbf{\sigma}' \right)
    \right)\\
    & = & 2 \mu_{\tmmathbf{h}} \left( \tmmathbf{\sigma}(t) = + \right)
    \mu_{\tmmathbf{h}} \left( \tmmathbf{\sigma}(t) = - \right) \, \left[
    \mu_{\tmmathbf{h}} \left( f|\tmmathbf{\sigma}(t) = + \right) -
    \mu_{\tmmathbf{h}} \left( f|\tmmathbf{\sigma}(t) = - \right) \right]
  \end{eqnarray*}
  Conditionally on $\tmmathbf{\sigma} (t)$, the spin on $[0, t]$ and $[t,
  \beta]$ are independent. So if we take the product of two couplings as in
  Lemma \ref{lem:coupling} on these two time intervals, we obtain a monotone
  coupling $\Psi$ of $\mu \left( .|\tmmathbf{\sigma}(t) = + \right)$ and $\mu
  \left( .|\tmmathbf{\sigma}(t) = - \right)$ with the property that
  \begin{gather*}
    \Psi (\tmmathbf{\sigma}^+ = +, \tmmathbf{\sigma}^- = -)  = 
    \prod_{\varepsilon = \pm} \mu_{t, \tmmathbf{h}, \lambda}
    (\tmmathbf{\sigma}= \varepsilon | \sigma (t) = \varepsilon) \mu_{\left( t,
    \beta \right), \tmmathbf{h}, \lambda} (\tmmathbf{\sigma}= \varepsilon |
    \sigma (t) = \varepsilon)\\
     =  \mu \left( \tmmathbf{\sigma}= + |\tmmathbf{\sigma}(t) = + \right)
    \mu \left( \tmmathbf{\sigma}= - |\tmmathbf{\sigma}(t) = - \right) .
  \end{gather*}
  Furthermore,
  \begin{eqnarray*}
    \mu \left( f|\tmmathbf{\sigma}(t) = + \right) - \mu \left(
    f|\tmmathbf{\sigma}(t) = - \right) & = & \Psi (f (\tmmathbf{\sigma}^+) - f
    (\tmmathbf{\sigma}^-))\\
    & \geqslant & \Psi (\tmmathbf{\sigma}^+ = +, \tmmathbf{\sigma}^- = -)  (f
    ( +) - f (-))
  \end{eqnarray*}
  so we conclude that
  \begin{eqnarray}
    \tmop{Cov}_{\mu_{\tmmathbf{h}}} (\tmmathbf{\sigma}(t), f
    (\tmmathbf{\sigma})) & \geqslant & 2 \mu_{\tmmathbf{h}} \left(
    \tmmathbf{\sigma}= + \right) \mu_{\tmmathbf{h}} \left( \tmmathbf{\sigma}=
    - \right) (f ( +) - f (-))  \label{covf}
  \end{eqnarray}
  which, in combination with (\ref{eps}), proves the statement
  (\ref{mu:stpos}) when $\pi = \emptyset$. It is not difficult to derive an
  explicit constant as $\varphi \left( \tmmathbf{\sigma}= + \right) = \varphi
  \left( \tmmathbf{\sigma}= - \right) = \exp \left( - \beta \lambda \right) /
  2$, the density of $\mu$ with respect to $\varphi$ being bounded. Now we
  address the case of periodic imaginary time boundary conditions. We call
  $\mu^{\varepsilon, \varepsilon'}_{\tmmathbf{h}} = \mu_{\tmmathbf{h}} \left(
  \cdummy | \tmmathbf{\sigma}(0) = \varepsilon, \tmmathbf{\sigma}(\beta) =
  \varepsilon' \right)$ the measure $\mu_{\tmmathbf{h}}$ conditioned on
  $\tmmathbf{\sigma} (0) = \varepsilon$ and $\tmmathbf{\sigma} (\beta) =
  \varepsilon'$. It is clear that
  \begin{eqnarray}
    \mu_{\tmmathbf{h}}^{\tmop{per}} & = & p \mu_{\tmmathbf{h}}^{+ +} + \left(
    1 - p \right) \mu_{\tmmathbf{h}}^{- -} .  \label{dec:muper}
  \end{eqnarray}
  where
  \begin{eqnarray*}
    p \nocomma \nocomma \nocomma = p_{\tmmathbf{h}} & = &
    \frac{\mu_{\tmmathbf{h}} \left( f^+ \right)}{\mu_{\tmmathbf{h}} \left( f^+
    \right) + \mu_{\tmmathbf{h}} \left( f^- \right)} .
  \end{eqnarray*}
  and $f^{\varepsilon} (\sigma) =\tmmathbf{1}_{\{\sigma (0) = \sigma (\beta) =
  \varepsilon\}}$, for $\varepsilon = \pm$. Note that $f^+$ and $- f^-$ are
  increasing, and their amplitude is $\varepsilon f^{\varepsilon} \left( +
  \right) - \varepsilon f^{\varepsilon} \left( - \right) = 1$. When we take
  derivatives, it is a consequence of (\ref{mu:stpos}) for free imaginary time
  boundary condition that
  \begin{eqnarray}
    \left[ \frac{\mathd p_{\tmmathbf{h}+ s \Delta \tmmathbf{h}}}{\mathd s}
    \right]_{s = 0} & = & \left[ \frac{\mathd}{\mathd s} \mu_{\tmmathbf{h}+ s
    \Delta \tmmathbf{h}} (f^+) \right]_{s = 0} \times \frac{\mu_{\tmmathbf{h}}
    \left( f^- \right)}{\left[ \mu_{\tmmathbf{h}} \left( f^+ \right) +
    \mu_{\tmmathbf{h}} \left( f^- \right) \right]^2} \nonumber\\
    &  & + \left[ \frac{\mathd}{\mathd s} \mu_{\tmmathbf{h}+ s \Delta
    \tmmathbf{h}} (- f^-) \right]_{s = 0} \times \frac{\mu_{\tmmathbf{h}}
    \left( f^+ \right)}{\left[ \mu_{\tmmathbf{h}} \left( f^+ \right) +
    \mu_{\tmmathbf{h}} \left( f^- \right) \right]^2} \nonumber\\
    & \geqslant & c \frac{\int_0^{\beta} \mathd t \Delta
    \tmmathbf{h}(t)}{\mu_{\tmmathbf{h}} (\tmmathbf{\sigma}(0)
    =\tmmathbf{\sigma}(\beta))} .  \label{dp}
  \end{eqnarray}
  Now we take derivatives in (\ref{dec:muper}):
  \begin{eqnarray}
    \left[ \frac{\mathd}{\mathd s} \mu_{\tmmathbf{h}+ s \Delta
    \tmmathbf{h}}^{\tmop{per}} (f) \right]_{s = 0} & = & \left[ \frac{\mathd
    p_{\tmmathbf{h}+ s \Delta \tmmathbf{h}}}{\mathd s} \right]_{s = 0} \left(
    \mu_{\tmmathbf{h}}^{+ +} \left( f \right) - \mu_{\tmmathbf{h}}^{- -}
    \left( f \right) \right) \nonumber\\
    &  & + p \frac{\mathd}{\mathd s} \mu_{\tmmathbf{h}+ s \Delta
    \tmmathbf{h}}^{+ +} (f) + \left( 1 - p \right) \frac{\mathd}{\mathd s}
    \mu_{\tmmathbf{h}+ s \Delta \tmmathbf{h}}^{- -} (f) .  \label{dmuper}
  \end{eqnarray}
  As the imaginary time boundary condition can be realized by adding to
  $\tmmathbf{h}$ an additional field $\varepsilon A\tmmathbf{k}_{\delta}$,
  where $\tmmathbf{k}_{\delta} =\tmmathbf{1}_{[0, \delta] \cup [\beta -
  \delta, \beta]}$, in the limit $A \rightarrow + \infty$ and $\delta
  \rightarrow 0$ (in that order), the last two terms are non-negative
  according to (\ref{mu:stpos}) for $\pi = \emptyset$. So in order to conclude
  the proof of (\ref{mu:stpos}) for periodic imaginary time boundary condition
  it is enough to provide a lower bound on $\mu_{\tmmathbf{h}}^{+ +} \left( f
  \right) - \mu_{\tmmathbf{h}}^{- -} \left( f \right)$. We know that
  \begin{eqnarray*}
    \mu_{\tmmathbf{h}}^{\varepsilon \varepsilon} \left( f \right) & = &
    \lim_{\delta \rightarrow 0} \lim_{A \rightarrow + \infty}
    \mu_{\tmmathbf{h}+ \varepsilon A\tmmathbf{k}_{\delta}} \left( f \right) .
  \end{eqnarray*}
  Furthermore, according to (\ref{mu:stpos}) for $\pi = \emptyset$, for any
  $\delta > 0$, $A \in {\mathbf R} \mapsto \mu_{\tmmathbf{h}+
  A\tmmathbf{k}_{\delta}} \left( f \right)$ is a increasing function and, on
  the interval $\left[ - 1 / \delta, 1 / \delta \right]$ its derivative is not
  smaller than $c' \delta \left( f \left( + \right) - f \left( - \right)
  \right)$. This proves that
  \begin{eqnarray}
    \lim_{A \rightarrow + \infty} \mu_{\tmmathbf{h}+ A\tmmathbf{k}_{\delta}}
    \left( f \right) - \lim_{A \rightarrow + \infty} \mu_{\tmmathbf{h}-
    A\tmmathbf{k}_{\delta}} \left( f \right) & \geqslant & c' \left( f \left(
    + \right) - f \left( - \right) \right)  \label{eq17}
  \end{eqnarray}
  for any $\delta > 0$, where $c'$ does not depend on $\delta$, and
  consequently $\mu_{\tmmathbf{h}}^{+ +} \left( f \right) -
  \mu_{\tmmathbf{h}}^{- -} \left( f \right)$ satisfies the same lower bound.
  Putting (\ref{eq17}) with (\ref{dp}) into (\ref{dmuper}) we obtain
  (\ref{mu:stpos}) for periodic imaginary time boundary condition.
\end{proof}

\begin{proof}[Proof of Corollary \ref{cor:FKG}] Fix $\varepsilon\ll 1$
  and assume first that $f, g$ are increasing functions
  of a single spin, measurable w.r.t.~$\mathcal F_{T + \varepsilon}$, for some $T \in \left[ 0,
  \beta \right]$, where $\mathcal F_t$ is the $\sigma$-algebra
  generated by $\{\sigma(s)\}_{s\le t}$.  We prove that
         \begin{eqnarray}
          \tmop{Cov}\left( f g |\mathcal F_T\right)
    & \geqslant & - C \varepsilon^2 \left\| f \right\|_{\infty} \left\|
    g \right\|_{\infty} 
      \label{FKGeps}
  \end{eqnarray}
  where
  \begin{eqnarray}
     \tmop{Cov}\left( f, g |\mathcal F_T\right) & =& \mu \left( fg |\mathcal F_T\right)
    - \mu \left( f |\mathcal F_T \right) \mu \left( g |\mathcal F_T\right). \nonumber
      \end{eqnarray} 
  Indeed, call $\nu$ the distribution $\mu$ conditioned to $\{\tmmathbf{\sigma}
  \left( t \right)\}_{ t \leqslant T}$ and to the event that $\tmmathbf{\sigma}$
  has at most one flip in $\left[ T, T + \varepsilon \right]$. The
  distribution $\nu$ is completely described by the law of the time of the
  flip ($+ \infty$ if no flip). But $f$ and $g$ are both monotone functions of
  this random time (both increasing or decreasing), so it follows from the FKG
  inequality for distributions on the real line (Lemma 16.2 in the lectures
  notes by Peres {\cite{PeresUBC}}) that $\nu \left( fg \right) \geqslant \nu
  \left( f \right) \nu \left( g \right)$. But the total variation distance
  between $\nu$ and $\mu \left( . | \tmmathbf{\sigma} \left( t \right), t
  \leqslant T \right)$ is less than $C \varepsilon^2$, which proves
  (\ref{FKGeps}).
  
  Now we take two arbitrary increasing functions  $f, g$  of a single spin and
  choose $T =  \beta - \varepsilon$. When we apply the standard formula
  for conditional covariance together with (\ref{FKGeps}) we get that
  \begin{eqnarray*}
    \tmop{Cov}\left( f ,g \right)   & = &
        \mu \left( \tmop{Cov}(f,g | \mathcal F_T)\right) 
       + \tmop{Cov}\left(\mu \left( f| \mathcal F_T\right), \mu \left( g|
           \mathcal F_T\right) \right)\\
& \geqslant & - C \varepsilon^2 \left\| f \right\|_{\infty} \left\| g
    \right\|_{\infty} + \tmop{Cov}\left(\mu \left( f| \mathcal F_T\right), \mu \left( g|
           \mathcal F_T\right) \right)
 \end{eqnarray*}
  where $\mu \left( f | \mathcal F_T \right)$ and $\mu \left( g | \mathcal F_T \right)$ are increasing functions with infinite norm less than that of
  $\left\| f \right\|_{\infty}$ and $\left\| g \right\|_{\infty}$,
  respectively. By applying (\ref{FKGeps}) repeatedly with $T = \beta - k
  \varepsilon$, $k=2,\dots \beta/\varepsilon$ we conclude that
  \begin{eqnarray*}
   \tmop{Cov}\left( f ,g \right)
    &  \geqslant & - C \beta \varepsilon \left\| f \right\|_{\infty} \left\| g
    \right\|_{\infty}
  \end{eqnarray*}
  and the claim follows by letting $\varepsilon \rightarrow 0$.
\end{proof}

\section{Glauber dynamics for the Quantum Ising model}

\subsection{Definition of the generator and the semi-group}

Here we define the Glauber dynamics for finite graphs and establish some
preliminary properties.

The dynamics consists in resampling spins locally, according to the field
generated by their neighbors. Given the graph $G$ and the parameters $\beta>0, \lambda \geqslant 0$, $\tmmathbf{h}:V\times[0,\beta]\mapsto \mathbf R$ integrable, we call $\mu = \mu_{G ; \beta, \tmmathbf{h}, \lambda}$
the Gibbs measure on $G$ with corresponding parameters and
\begin{eqnarray*}
  \mu_x^{\tmmathbf{\rho}} & = & \mu \left( . | \tmmathbf{\sigma}_y
  =\tmmathbf{\rho}_y, \forall y \in V \setminus \left\{ x \right\} \right)
  \text{, \ \ } \forall \tmmathbf{\rho} \in \Sigma^V, \forall x \in V.
\end{eqnarray*}
Note that $\mu_x$ takes into account both the field $\tmmathbf{h}$ and the
boundary condition on $V$. According to the DLR equation,
$\mu^{\tmmathbf{\rho}}_x$ is the measure obtained by taking
$\tmmathbf{\sigma}=\tmmathbf{\rho}$ on $V \setminus \left\{ x \right\}$ and
$\tmmathbf{\sigma}_x$ according to the Gibbs measure on $\left\{ x \right\}$
with field $\tmmathbf{h}_x + \sum_{y \sim x} \tmmathbf{\rho}_y$. We can
interpret $\mu_x$ as a kernel, since for each $\tmmathbf{\rho} \in \Sigma^V$
and $x \in V$, $\mu^{\tmmathbf{\rho}}_x$ is a probability measure
(furthermore, $\tmmathbf{\rho} \mapsto \mu^{\tmmathbf{\rho}}_x$ is continuous,
cf.~(\ref{mu:Gamma})). Because $\mu_x^{\tmmathbf{\rho}}$ is a conditional
expectation, it is a contraction in $L^2 \left( \mu \right)$.

Next we define the generator of the Glauber dynamics by
\begin{eqnarray*}
  \mathcal{L} & = & \sum_{x \in V} \left( \mu_x - I \right)
\end{eqnarray*}
where $I$ is the identity operator. This is clearly a bounded operator on $L^2
\left( \mu_{G ; \beta, \tmmathbf{h}, \lambda} \right)$ and the associated
Markov semi-group is
\begin{eqnarray}
  P_t  &=& e^{t\mathcal{L}} 
 \nonumber\\
  & = & \sum_{n \geqslant 0} e^{- t \left| V \right|}  \frac{\left( t \left|
  V \right| \right)^n}{n!} \times \left( \frac{1}{\left| V \right|} \right)^n
  \sum_{x_1, \ldots, x_n \in V} \mu_{x_1} \cdots \mu_{x_n} .  \label{Pt}
\end{eqnarray}
According to its definition, $P_t$ is a bounded operator on $L^2 \left( \mu
\right)$. Equation (\ref{Pt}) shows as well that, for any $t \geqslant 0$,
$P_t$ is a convex combination of the iterates of $\mu_x$, and is therefore a
Markov kernel that contracts $L^2$. For any $\tmmathbf{\rho} \in \Sigma^V$, we
will write $P_t^{\tmmathbf{\rho}}$ for the probability measure on $\Sigma^V$
which, on bounded functions $f$, acts as $P_t^{\tmmathbf{\rho}} \left( f
\right) = \left( P_t f \right) \left( \tmmathbf{\rho} \right)$. If
$\tmmathbf{\rho}$ is distributed according to a probability measure $\nu$ on
$\Sigma$, we will write $\nu P_t$ for the measure $\int P_t^{\tmmathbf{\rho}}
\mathd \nu \left( \tmmathbf{\rho} \right)$.

\subsection{Mixing time and spectral gap}

In this section, in analogy with the classical situation, we prove some basic
results that are useful to control the relaxation of the dynamics to the
equilibrium Gibbs measure.

\begin{proposition}
  Let some $G = \left( V, E \right)$ finite, $\beta \geqslant 0$,
  $\tmmathbf{h}: V \times \left[ 0, \beta \right] \rightarrow {\mathbf R}$ integrable and
  $\lambda \geqslant 0$. Define
  \begin{eqnarray*}
    T_{\tmop{mix}} & = & \inf \left\{ t \geqslant 0 : \forall \tmmathbf{\rho},
    \tmmathbf{\eta} \in \Sigma^V, \left\| P_t^{\tmmathbf{\rho}} -
    P_t^{\tmmathbf{\eta}} \right\|_{\tmop{TV}} \leqslant e^{- 1} \right\} .
  \end{eqnarray*}
  Then $T_{\tmop{mix}} < \infty$ and, for any $t \geqslant 0$,
  \[ \sup_{\tmmathbf{\rho} \in \Sigma^V} \left\| P_t^{\tmmathbf{\rho}} - \mu
     \right\|_{\tmop{TV}} \leqslant \sup_{\tmmathbf{\rho}, \tmmathbf{\eta} \in
     \Sigma^V} \left\| P_t^{\tmmathbf{\rho}} - P_t^{\tmmathbf{\eta}}
     \right\|_{\tmop{TV}} \leqslant e^{- \left\lfloor t / T_{\tmop{mix}}
     \right\rfloor} . \]
\end{proposition}

\begin{proof}
  The first inequality is a consequence of $\mu$ being invariant by $P_t$ (in
  other words, $\mu$ is a convex combination of the $P_t^{\tmmathbf{\eta}}$).
  The second inequality is classical consequence of
  \begin{eqnarray*}
    \bar{d} \left( t \right) & = & \sup_{\tmmathbf{\rho}, \tmmathbf{\eta} \in
    \Sigma^V} \left\| P_t^{\tmmathbf{\rho}} - P_t^{\tmmathbf{\eta}}
    \right\|_{\tmop{TV}}
  \end{eqnarray*}
  being sub-multiplicative, cf.~{\cite{LPW09}}. It remains to prove that
  $T_{\tmop{mix}} < \infty$, or equivalently that $\bar{d} \left( t \right) <
  1$ for some $t > 0$. This follows from (\ref{Pt}) once we remark that any
  $\mu_{x_1} \cdots \mu_{x_n}$ with $\left\{ x_1, \ldots, x_n \right\} = V$
  gives a probability at least $c^{- \left| V \right|}$ to the uniform plus
  state, uniformly in the starting state $\tmmathbf{\rho}$.
\end{proof}

\begin{proposition}
  Let some $G = \left( V, E \right)$ finite, $\beta \geqslant 0$,
  $\tmmathbf{h}: V \times \left[ 0, \beta \right] \rightarrow {\mathbf R}$ integrable and
  $\lambda \geqslant 0$. Define
  \begin{eqnarray*}
    \tmop{gap(\mathcal L)} & = & \inf_{f \in L^2 \left( \mu \right) : \tmop{Var} \left( f
    \right) > 0} \frac{\tmop{Cov} \left( f, -\mathcal{L}f \right)}{\tmop{Var}
    \left( f^2 \right)}
  \end{eqnarray*}
  where $\tmop{Cov}$ and $\tmop{Var}$ refer, respectively, to the covariance
  and the variance under $\mu$. Then
  \begin{enumerate}
    \item There exists $c < \infty$ depending on $\beta, \lambda$ and $\left\|
    \tmmathbf{h} \right\|_{\infty}$ such that
    \begin{eqnarray*}
      \tmop{gap (\mathcal L)} & \geqslant & c^{- \left| V \right|} .
    \end{eqnarray*}
    \item For any $f \in L^2 \left( \mu \right)$,
    \begin{eqnarray}
      \tmop{Var} \left( P_t f \right) & \leqslant & e^{- 2 t \tmop{gap(\mathcal L)}}
      \tmop{Var} \left( f \right) .  \label{var:gap}
    \end{eqnarray}
  \end{enumerate}
\end{proposition}

\begin{proof}
  The proof of the second point is standard. For the first one we refer the
  reader to the proof of Theorem 6.4 in the Saint Flour course {\cite{Ma97}}.
\end{proof}

\subsection{Monotonicity}

Now we address the question of the monotonicity of the dynamics. An immediate
consequence of (\ref{Pt}) together with the monotonicity of the single site
measure (Theorem \ref{thm:mono}) is the following fact:

\begin{proposition}
  Take $\tmmathbf{\rho}, \tmmathbf{\eta} \in \Sigma^V$ such that
  $\tmmathbf{\rho} \leqslant \tmmathbf{\eta}$, and $\tmmathbf{h} \leqslant
  \tilde{\tmmathbf{h}}$. Denote by $\tilde{P}_t$ the semi-group corresponding
  to field $\tilde{\tmmathbf{h}}$. Then,
  \begin{eqnarray*}
    P_t^{\tmmathbf{\rho}} & \underset{\tmop{stoch} .}{\leqslant} &
    \tilde{P}_t^{\tmmathbf{\eta}} .
  \end{eqnarray*}
\end{proposition}

According to the convergence towards the equilibrium measure, it follows that

\begin{corollary} \label{cor:mu:mono}
  $\mu_{\tmmathbf{h}}$ increases stochastically with the field $\tmmathbf{h}$.
\end{corollary}

\begin{remark}
  The same argument as above could be used to establish the existence of a
  grand coupling, but for this we would need to know the existence of a grand
  coupling for the family of single spin measures, given an arbitrary family
  of external fields.
\end{remark}

\subsection{Censoring}

For any $A \subset V$, we let
\begin{eqnarray*}
  \mathcal{L}_A & = & \sum_{x \in A} \left( \mu_x - I \right) .
\end{eqnarray*}
Now we consider a function $A : {\mathbf R}^+ \rightarrow \mathcal{P} \left( V
\right)$ with finitely many discontinuities at $t_0 = 0 < t_1 < \ldots < t_n$.
We define the censored dynamics according to $A$ by the kernel
\begin{eqnarray}
  P_{A ; t} & = & e^{\left( t_1 - t_0 \right) \mathcal{L}_{A_0}} e^{\left( t_2
  - t_1 \right) \mathcal{L}_{A_1}} \cdots e^{\left( t - t_k \right)
  \mathcal{L}_{A_k}}  \label{PAt}
\end{eqnarray}
where $k$ is the largest integer in $\left\{ 0, \ldots, n \right\}$ such that
$t \geqslant t_k$, and $A_i = A \left( t_i^+ \right)$. Of course, when $A
\left( t \right) = V$ for any $t \geqslant 0$ we get $P_{A ; t} = P_t$, the
uncensored dynamics. The theory of censoring due to Peres and Winkler
{\cite{PeresUBC}} also applies here. Remarkably their result on total
variation extends also to variance and entropy.

\begin{proposition}
  \label{prop:censoring}Consider $A, B : {\mathbf R}^+ \rightarrow \mathcal{P}
  \left( V \right)$ as above. Assume that, for any $t \geqslant 0$, $A \left(
  t \right) \subset B \left( t \right)$. Assume that $\nu$ is absolutely
  continuous with respect to $\mu$ with $\mathd \nu / \mathd \mu\in
  L^2(d\mu)$ and increasing.\mnote{Fabio}
  Then, for any $t \geqslant 0$, both $\nu P_{A ; t}$ and $\nu P_{B ; t}$ are
  absolutely continuous with respect to $\mu$, their Radon-Nikodym derivative
  is increasing and $\nu P_{B ; t} \prec \nu P_{A ; t}$. Moreover the following
  inequalities hold:
  \begin{eqnarray}
    \tmop{Var} \left( \frac{\mathd \left( \nu P_{B ; t} \right)}{\mathd \mu}
    \right) & \leqslant & \tmop{Var} \left( \frac{\mathd \left( \nu P_{A ; t}
    \right)}{\mathd \mu} \right)  \label{cens:var}\\
    \tmop{Ent} \left( \frac{\mathd \left( \nu P_{B ; t} \right)}{\mathd \mu}
    \right) & \leqslant & \tmop{Ent} \left( \frac{\mathd \left( \nu P_{A ; t}
    \right)}{\mathd \mu} \right)  \label{cens:ent}\\
    \left\| \nu P_{B ; t} - \mu \right\|_{\tmop{TV}} & \leqslant & \left\| \nu
    P_{A ; t} - \mu \right\|_{\tmop{TV}} .  \label{cens:tv}
  \end{eqnarray}
\end{proposition}

\begin{remark}
A special case satisfying the assumptions of the proposition is when
$\nu $ is concentrated on the identical equal to $+$
configuration. In that case we will write $\nu P_{A ; t} = P_{A ; t}^+$
  and $\nu P_{B ; t} = P_{B ; t}^+$. 
\end{remark}

Following {\cite{PeresUBC}} we begin the proof with two lemmas.

\begin{lemma}
  \label{lem:PW}Consider $\nu$ some measure on $\Sigma^V$, and assume that
  $\nu$ is absolutely continuous with respect to $\mu$ with $\frac{\mathd
  \nu}{\mathd \mu}$ being an increasing function. Then $\nu \mu_x$ is
  absolutely continuous with respect to $\mu$ and its Radon-Nikodym derivative
  is increasing as well.
\end{lemma}

\begin{proof}
  Let $\tmmathbf{\sigma} \leqslant \tmmathbf{\tau}$ and assume that $\mathd
  \nu / \mathd \mu$ is increasing. We denote by $\tmmathbf{\sigma}^{\star}$
  the spin configuration on $V \setminus \left\{ x \right\}$ equal to
  $\tmmathbf{\sigma}$ on $V \setminus \left\{ x \right\}$, and by
  $\tmmathbf{\sigma}_x^{\tmmathbf{\xi}}$ the spin configuration on $V$ equal
  to $\tmmathbf{\xi}$ at $x$ and to $\tmmathbf{\sigma}$ on $V \setminus
  \left\{ x \right\}$. Then
  \begin{eqnarray*}
    \frac{\mathd \left( \nu \mu_x \right)}{\mathd \mu} (\tmmathbf{\sigma}) & =
    & \frac{\mathd \nu}{\mathd \mu} (\tmmathbf{\sigma}^{\star}) = \int \mu
    (\mathd \tmmathbf{\xi}|\tmmathbf{\sigma}^{\star}) \frac{\mathd \nu}{\mathd
    \mu} (\tmmathbf{\sigma}_x^{\tmmathbf{\xi}})\\
    & \leqslant & \int \mu (\mathd \tmmathbf{\xi}|\tmmathbf{\sigma}^{\star})
    \frac{\mathd \nu}{\mathd \mu} (\tmmathbf{\tau}_x^{\tmmathbf{\xi}})\\
    & \leqslant & \int \mu (\mathd \tmmathbf{\xi}|\tmmathbf{\tau}^{\star})
    \frac{\mathd \nu}{\mathd \mu} (\tmmathbf{\tau}_x^{\tmmathbf{\xi}}) =
    \frac{\mathd \left( \nu \mu_x \right)}{\mathd \mu} \left( \tmmathbf{\tau}
    \right)
  \end{eqnarray*}
  where we used, in the first line, the fact that the density does not depend
  on $\tmmathbf{\sigma}_x$ since that spin is resampled, and at the second
  line the assumption that $\mathd \nu / \mathd \mu$ is increasing, and
  finally the fact that the single spin marginal increases with the external
  field (Theorem \ref{thm:mono}).
\end{proof}

\begin{lemma}
  \label{lem:mx:dec}Consider $\nu$ some measure on $\Sigma^V$, and assume that
  $\nu$ is absolutely continuous with respect to $\mu$ with $\frac{\mathd
  \nu}{\mathd \mu}$ being an increasing function. Then $\nu \mu_x \prec \nu$.
\end{lemma}

\begin{proof}
  Contrary to {\cite{PeresUBC}} the set of single spins configurations is not
  totally ordered. For this reason we use an alternative proof based on the
  FKG inequality for single spin measures. Let $f$ increasing. We start with
  \begin{eqnarray*}
    \nu \left( f \right) - \left( \nu \mu_x \right) \left( f \right) & = & \nu
    \left( f - \mu_x \left( f \right) \right)
  \end{eqnarray*}
  and condition on $\tmmathbf{\sigma}^{\star}$, the spin configuration outside
  $x$. We have
  \begin{eqnarray*}
    \nu \left( f - \mu_x \left( f \right) \left| \tmmathbf{\sigma}^{\star}
    \right. \right) & = & \mu_x^{\tmmathbf{\sigma}^{\star}} \left( \left[
    \frac{\frac{\mathd \nu}{\mathd \mu}}{\mu_x^{\tmmathbf{\sigma}^{\star}}
    \left( \frac{\mathd \nu}{\mathd \mu} \right)} - 1 \right] f \right)\\
    & \geqslant & \mu_x^{\tmmathbf{\sigma}^{\star}} \left( \frac{\frac{\mathd
    \nu}{\mathd \mu}}{\mu_x^{\tmmathbf{\sigma}^{\star}} \left( \frac{\mathd
    \nu}{\mathd \mu} \right)} - 1 \right) \mu_x^{\tmmathbf{\sigma}^{\star}}
    \left( f \right)\\
    & = & 0
  \end{eqnarray*}
  where in the second line we use the FKG inequality for a single spin
  (Corollary \ref{cor:FKG}).
\end{proof}

\begin{proof}
  (Proposition \ref{prop:censoring}). Lemmas \ref{lem:PW} and \ref{lem:mx:dec}
  together with formula (\ref{Pt}) imply that $\nu P_{A ; t}$ and $\nu P_{B ;
  t}$ are absolutely continuous with respect to $\mu$, that their
  Radon-Nikodym derivative is increasing and also that $\nu P_{B ; t}
  \prec \nu P_{A ; t}$. Now we prove the inequalities. For the
  variance, we remark that

  \begin{eqnarray*}
    \tmop{Var} \left( \frac{\mathd \left( \nu P_{B ; t} \right)}{\mathd \mu}
    \right) & = & \mu \left( \left( \frac{\mathd \left( \nu P_{B ; t}
    \right)}{\mathd \mu} \right)^2 \right) - 1 = \nu P_{B ; t} \left(
    \frac{\mathd \left( \nu P_{B ; t} \right)}{\mathd \mu} \right) - 1\\
    & \leqslant & \nu P_{A ; t} \left( \frac{\mathd \left( \nu P_{B ; t}
    \right)}{\mathd \mu} \right) - 1 = \tmop{Cov} \left( \frac{\mathd \left(
    \nu P_{A ; t} \right)}{\mathd \mu}, \frac{\mathd \left( \nu P_{B ; t}
    \right)}{\mathd \mu} \right)\\
    & \leqslant & \tmop{Var} \left( \frac{\mathd \left( \nu P_{A ; t}
    \right)}{\mathd \mu} \right)^{1 / 2} \tmop{Var} \left( \frac{\mathd \left(
    \nu P_{B ; t} \right)}{\mathd \mu} \right)^{1 / 2}
  \end{eqnarray*}
  which proves (\ref{cens:var}). For the entropy we recall that
  \begin{eqnarray*}
    \tmop{Ent} \left( f \right) & = & \sup \left\{ \mu \left( fg \right) : g
    \text{ with } \mu \left( e^g \right) = 1 \right\}
  \end{eqnarray*}
  therefore
  \begin{eqnarray*}
    \tmop{Ent} \left( \frac{\mathd \left( \nu P_{B ; t} \right)}{\mathd \mu}
    \right) & = & \left( \nu P_{B ; t} \right) \left( \log \frac{\mathd \left(
    \nu P_{B ; t} \right)}{\mathd \mu} \right)\\
    & \leqslant & \left( \nu P_{A ; t} \right) \left( \log \frac{\mathd
    \left( \nu P_{B ; t} \right)}{\mathd \mu} \right) = \mu \left(
    \frac{\mathd \left( \nu P_{A ; t} \right)}{\mathd \mu} \log \frac{\mathd
    \left( \nu P_{B ; t} \right)}{\mathd \mu} \right)\\
    & \leqslant & \tmop{Ent} \left( \frac{\mathd \left( \nu P_{A ; t}
    \right)}{\mathd \mu} \right) .
  \end{eqnarray*}
  Finally we recall for completeness the proof of (\ref{cens:tv}):
  \begin{eqnarray*}
    \left\| \nu P_{B ; t} - \mu \right\|_{\tmop{TV}} & = & \nu P_{B ; t}
    \left( \frac{\mathd \left( \nu P_{B ; t} \right)}{\mathd \mu} \geqslant 1
    \right) - \mu \left( \frac{\mathd \left( \nu P_{B ; t} \right)}{\mathd
    \mu} \geqslant 1 \right)\\
    & \leqslant & \nu P_{A ; t} \left( \frac{\mathd \left( \nu P_{B ; t}
    \right)}{\mathd \mu} \geqslant 1 \right) - \mu \left( \frac{\mathd \left(
    \nu P_{B ; t} \right)}{\mathd \mu} \geqslant 1 \right)\\
    & \leqslant & \left\| \nu P_{A ; t} - \mu \right\|_{\tmop{TV}} .
  \end{eqnarray*}
\end{proof}

\section{Ising model on regular trees}

\subsection{Notation and main results}

In the following we specialize to the case where the underlying graph is
$\mathbbm{T}_l^b$, the rooted regular tree with $b \geqslant 2$ children at
each node except the leaves, and depth $l \geqslant 0$. We always denote by
$r$ the root of the tree. When $z \in \mathbbm{T}_l^b$ we denote by $\left| z
\right|$ the depth of $z$, that is the graph distance to the root. We say that
$z$ is a leaf if $\left| z \right| = l$. Given $\tmmathbf{\tau} \in
\Sigma^{\mathbbm{T}_{\infty}^b}$ and $A$ is a subset of
$\mathbbm{T}_{\infty}^b$ we define $\mu^{\tmmathbf{\tau}}_A$ as the Gibbs
measure on $\Sigma^A$ with boundary condition $\tmmathbf{\tau}$ acting as an
additional local field at $z \in A$ given by $\sum_{y \nin A, y \sim z}
\tmmathbf{\tau} \left( y \right)$, where $y \sim z$ means that $y, z$ are
neighbors. When $\tmmathbf{\tau}$ is identically equal to plus/minus we will
simply write $\mu^{\pm}_A$. When $A =\mathbbm{T}_l^b$, we will write
$\mu_l^{\tmmathbf{\tau}}$.

\begin{definition}
  We say that the parameters $\beta, \lambda \geqslant 0$, $\tmmathbf{h} \in
  L^1 \left( \left[ 0, \beta \right] \right)$ are inside the uniqueness region
  if the boundary condition on the leaves of $\mathbbm{T}_l^b$ does not change
  the marginal distribution at the root in the limit $l \rightarrow + \infty$.
  Because of stochastic domination, this is equivalent to
  \begin{eqnarray}
    \lim_{l \rightarrow + \infty} \left\| \mu_l^+ \left( \tmmathbf{\sigma}_r
    \in \cdummy \right) - \mu_l^- \left( \tmmathbf{\sigma}_r \in \cdummy
    \right) \right\|_{\tmop{TV}} & = & 0.  \label{hyp:bc:ni}
  \end{eqnarray}
\end{definition}

Next, following {\cite{MSW04}} we define the exponents $\gamma$ and $\kappa$
as follows:

\begin{definition}
  Given $\beta, \lambda \geqslant 0$, $\tmmathbf{h} \in L^1 \left( \left[ 0,
  \beta \right] \right)$ we let
  \begin{eqnarray}
    \gamma & = & \sup_l \max \left\| \mu_A^{\tmmathbf{\eta}} \left(
    \tmmathbf{\sigma}_z \in \cdummy \right) - \mu_A^{\bar{\tmmathbf{\eta}}}
    \left( \tmmathbf{\sigma}_z \in \cdummy \right) \right\|_{\tmop{TV}} 
    \label{gamma}
  \end{eqnarray}
  where the maximum is taken over all subsets $A \subset \mathbbm{T}_l^b$, all
  vertices $y$ on the external boundary of $A$, all boundary configurations
  $\tmmathbf{\eta}, \bar{\tmmathbf{\eta}} \in \Sigma^{\mathbbm{T}_{\infty}^b}$
  that differ only at $y$ and all neighbors $z \in A$ of $y$.
  
  Given a boundary condition $\tmmathbf{\tau} \in
  \Sigma^{\mathbbm{T}_{\infty}^b}$, we let $\kappa \left( \tmmathbf{\tau}
  \right)$ be the infimum of $\kappa \geqslant 0$ such that there exists $l_0
  \geqslant 0$ such that, for any regular subtree $T \subset
  \mathbbm{T}_{\infty}^b$ with root $x$ and uniform depth, and for any $z \in
  T$ with $\left| z - x \right| \geqslant l_0$,
  \begin{eqnarray}
    \mu_T^{\tmmathbf{\tau}} \left( \tmmathbf{\sigma}_z \bullet
    1|\tmmathbf{\sigma}_x = + \right) - \mu_T^{\tmmathbf{\tau}} \left(
    \tmmathbf{\sigma}_z \bullet 1|\tmmathbf{\sigma}_x = - \right) & \leqslant
    & \kappa^{\left| z \right|} .  \label{kappa}
  \end{eqnarray}
\end{definition}

From Proposition \ref{prop:gamma} we know already that $\gamma < 1$. It
follows from a recursive coupling argument that $\kappa \leqslant \gamma$. In
complete analogy with the classical case, we prove the following results.

\begin{theorem}
  \label{thm:decay}(Decay of correlations) Let $\beta, \lambda \geqslant 0$,
  $\tmmathbf{h} \in L^1 \left( \left[ 0, \beta \right] \right)$.
  \begin{enumerate}
    \item $\kappa \left( + \right) \leqslant 1 / b$.
    
    \item If $\beta, \lambda, \tmmathbf{h}$ are in the uniqueness region, then
    for arbitrary $\tmmathbf{\tau} \in \Sigma^{\mathbbm{T}_{\infty}^b}$,
    $\kappa \left( \tmmathbf{\tau} \right) \leqslant 1 / b$.
  \end{enumerate}
\end{theorem}

\begin{theorem}
  \label{thm:mixing}(Fast mixing) Let $\beta, \lambda \geqslant 0$,
  $\tmmathbf{h} \in L^1 \left( \left[ 0, \beta \right] \right)$. Fix
  $\tmmathbf{\tau} \in \Sigma^{\mathbbm{T}_{\infty}^b}$ and assume that
  $\kappa \left( \tmmathbf{\tau} \right)$ is such that $\kappa \left(
  \tmmathbf{\tau} \right) \gamma b < 1$. Then the following holds.
  \begin{enumerate}
    \item The spectral gap of the Glauber dynamics on $\mathbbm{T}_l^b$ with
    boundary condition $\tmmathbf{\tau}$ is greater than a positive constant
    which does not depend on $l \geqslant 0$.
    
    \item The mixing time corresponding to the above dynamics is at most $Cl$
    where $C < \infty$ does not depend on $l \geqslant 0$.
  \end{enumerate}
\end{theorem}

\subsection{Conditional spin distributions}

Working with trees leads to major simplifications in the structure of the
Gibbs measure. In particular, given any $z \in \mathbbm{T}$, the restriction
of a Gibbs measure $\mu$ onto the subtrees of $z$, given
$\tmmathbf{\sigma}_z$, is a product measure. As a consequence, $\mu$ is fully
characterized by
\begin{enumerate}
  \item the marginal distribution $\mu \left( \tmmathbf{\sigma}_r \in .
  \right)$ of the spin at the root $r$,
  
  \item and the conditional marginal distributions $\mu \left(
  \tmmathbf{\sigma}_z \in . | \tmmathbf{\sigma}_{z^-} \right)$, $z \in
  \mathbbm{T} \setminus \left\{ r \right\}$, where $z^-$ denotes the ancestor
  of $z$.
\end{enumerate}
\begin{remark}
  The marginal distribution $\mu \left( \tmmathbf{\sigma}_r \in . \right)$ can
  be viewed as a conditional spin distribution if we add a ghost ancestor $0$
  to the root $r$ with constant spin $\tmmathbf{\sigma}_0 = 0$.
\end{remark}

For latter purposes we need that (conditional) measures form a vector space.
So we introduce the set $\mathcal{M}$ of finite signed measures on $(\Sigma,
\mathcal{B}(\Sigma))$, where $\mathcal{B} (\Sigma)$ is the Borel
$\sigma$-algebra associated to the Skorohod topology on $\Sigma$. Signed
measures have a Hahn-Jordan decomposition into their positive and negative
parts. This means that, for any $\mu \in \mathcal{M}$, there are two disjoint
sets $P$ and $N$ (unique up to $\mu$-negligible sets) such that $\mu$ gives
non-negative weight to every Borel subset of $P$, and non-positive weight to
every Borel subset of $N$. Let $\mu^+ = \mu (. \cap P)$ and $\mu^- = - \mu (.
\cap N)$. Then $\mu = \mu^+ - \mu^-$. We also consider $\left| \mu \right| =
\mu^+ + \mu^-$ , a positive measure, and recall that
\begin{eqnarray}
  \| \mu \|_{\tmop{TV}} & = & \frac{\left| \mu \right| (\Sigma)}{2} =
  \frac{1}{2} \sup_{f : |f| \leqslant 1} \mu (f) = \frac{1}{2} \sup_{A \subset
  \Sigma} \left( \mu \left( A \right) - \mu \left( A^c \right) \right) \label{TV}
\end{eqnarray}
defines a norm on $\mathcal{M}$, that turns $\mathcal{M}$ into a Banach space
as $\Sigma$ endowed with the Skorohod topology is a Polish space.

For convenience, we call $\mathcal{M}_0$ the subset of $\mathcal{M}$ made of
all finite signed measures $\mu$ with $\mu (\Sigma) = 0$. Both $\mathcal{M}$
and $\mathcal{M}_0$ are vector spaces. We also call $\mathcal{M}_+$ the set of
finite positive measures and $\mathcal{M}_{+, 1}$ the set of probability
measures on $\Sigma$.

Finally, we call $\mathcal{X}$ the set of functions from $\Sigma$ to
$\mathcal{M}$, $\mathcal{X}_0$ the set of functions from $\Sigma$ to
$\mathcal{M}_0$, and similarly $\mathcal{X}_{+, 1}$ for the set of functions
from $\Sigma$ to $\mathcal{M}_{+, 1}$. In particular, any marginal
(resp.~conditional marginal) distribution belongs to $\mathcal{M}_{+, 1}$
(resp.~$\mathcal{X}_{+, 1}$), and any difference of marginal
(resp.~conditional marginal) distributions belongs to $\mathcal{M}_0$
(resp.~$\mathcal{X}_0$). For notation consistance, for any $\rho \in
\mathcal{X}$ and any $\tmmathbf{\sigma} \in \Sigma$, we denote by
$\rho^{\tmmathbf{\sigma}} \left( . \right)$ the corresponding signed measure
on $\Sigma$. We will consider later on the problem of defining a norm on
$\mathcal{X}$ (see (\ref{Xinf}) and (\ref{X1})).

\subsection{The resampling operator and the cavity equation}

Given $\tmmathbf{\eta} \in L^1 \left( \left[ 0, \beta \right] \right)$, $\rho
\in \mathcal{M}$ and $\rho_1, \ldots, \rho_b \in \mathcal{X}$, we define
$R^{\tmmathbf{\eta}}_{\rho_1, \ldots, \rho_b} (\rho) \in \mathcal{M}$ by
\begin{eqnarray}
  R^{\tmmathbf{\eta}}_{\rho_1, \ldots, \rho_b} (\rho) & = & \int \rho (\mathd
  \tmmathbf{\sigma}_0) \rho_1^{\tmmathbf{\sigma}_0} (\mathd
  \tmmathbf{\sigma}_1) \cdots \rho_b^{\tmmathbf{\sigma}_0} (\mathd
  \tmmathbf{\sigma}_b) \mu_{\beta,
  \tmmathbf{h}+\tmmathbf{\eta}+\tmmathbf{\sigma}_1 + \cdots
  +\tmmathbf{\sigma}_b, \lambda} .  \label{R}
\end{eqnarray}
We give an interpretation to $R^{\tmmathbf{\eta}}_{\rho_1, \ldots, \rho_b}
(\rho)$ in an important special case.
\begin{enumerate}
  \item Consider some point of the tree $\mathbbm{T}$ with $b + 1$ neighbors.
  Conventionally we denote its spin by $\tmmathbf{\sigma}_0$. We call
  $\tmmathbf{\eta}$ the spin of its ancestor and $\tmmathbf{\sigma}_1, \ldots,
  \tmmathbf{\sigma}_b$ the spins of its children.
  
  \item Sample $\tmmathbf{\sigma}_0$ from $\rho$, the conditional distribution
  of $\tmmathbf{\sigma}_0$ given the neighbor spin $\tmmathbf{\eta}$.
  
  \item Sample each $\tmmathbf{\sigma}_i$ independently, according to
  $\rho_i^{\tmmathbf{\sigma}_0}$, the conditional distribution of
  $\tmmathbf{\sigma}_i$ given the parent spin $\tmmathbf{\sigma}_0$.
  
  \item Sample again $\tmmathbf{\sigma}_0$ according to the single site
  distribution with field $\tmmathbf{h}+\tmmathbf{\eta}+\tmmathbf{\sigma}_1 +
  \cdots +\tmmathbf{\sigma}_b$.
\end{enumerate}
If $\rho$ and the $\rho_i$ correspond to the conditional marginal
distributions of the Gibbs measure with field $\tmmathbf{h}$, then $\rho$ is
stable under the resampling operator according to the DLR equation. In other
words, $\rho = R^{\tmmathbf{\eta}}_{\rho_1, \ldots, \rho_b} (\rho)$ when
$\rho$ is the conditional distribution of $\tmmathbf{\sigma}_0$ given that the
spin above $0$ is $\tmmathbf{\eta}$, and the $\rho_i^{\tmmathbf{\sigma}_0}$
are the conditional distributions of $\tmmathbf{\sigma}_i$ given
$\tmmathbf{\sigma}_0$. More generally we have:

\begin{definition}
  Given $\rho_1, \ldots, \rho_b \in \mathcal{X}$, $\tmmathbf{\eta} \in L^1
  \left( \left[ 0, \beta \right] \right)$, we say that $\rho \in \mathcal{M}$
  satisfies the \tmtextit{cavity equation} with parameters $\rho_1, \ldots,
  \rho_b$ and $\tmmathbf{\eta}$ if
  \begin{eqnarray}
    \rho & = & R^{\tmmathbf{\eta}}_{\rho_1, \ldots, \rho_b} (\rho) 
    \label{cav}
  \end{eqnarray}
\end{definition}

\subsubsection{The resampling operator is a contraction}

From the definition (\ref{R}) it is clear that the applications $\rho \in
\mathcal{M} \mapsto R^{\tmmathbf{\eta}}_{\rho_1, \ldots, \rho_b} (\rho) \in
\mathcal{M}$ and $\rho_i \in \mathcal{X} \mapsto R^{\tmmathbf{\eta}}_{\rho_1,
\ldots, \rho_b} (\rho) \in \mathcal{M}$ are linear. Now we consider their
operator norm.

\begin{proposition}
  \label{prop:Scont}Let $\gamma$ be the constant in
  Proposition~\ref{prop:gamma} corresponding to the value of $M$ given
  by $M= \max \left( \lambda,
  \|\tmmathbf{h}\|_1 + \beta (b + 1) \right)$. Then,
  \begin{enumerate}
    \item For any $\rho_1, \ldots, \rho_b \in \mathcal{X}_{1, +}$, any
    $\tmmathbf{\eta} \in \Sigma$, the application
    $R^{\tmmathbf{\eta}}_{\rho_1, \ldots, \rho_b}$ is a contraction on
    $\mathcal{M}$, and a $\gamma$-contraction (i.e.~its operator norm is at
    most $\gamma$) when restricted to $\mathcal{M}_0$.
    
    \item For any $\rho_1, \ldots, \rho_b \in \mathcal{X}_{1, +}$, any
    $\tmmathbf{\eta} \in \Sigma$, the application $I -
    R^{\tmmathbf{\eta}}_{\rho_1, \ldots, \rho_b}$ is invertible on
    $\mathcal{M}_0$ and its inverse
    \begin{eqnarray}
      \left( \text{$I - R^{\tmmathbf{\eta}}_{\rho_1, \ldots, \rho_b}$}
      \right)^{- 1} & = & \sum_{k \geqslant 0} \left(
      R^{\tmmathbf{\eta}}_{\rho_1, \ldots, \rho_b} \right)^k  \label{R:inv}
    \end{eqnarray}
    has operator norm not larger than $\left( 1 - \gamma \right)^{- 1}$.
    
    \item For any $\rho \in \mathcal{M}$, $\rho_1, \ldots, \rho_b \in
    \mathcal{X}$,
    \begin{eqnarray}
      \left\| R^{\tmmathbf{\eta}}_{\rho_1, \ldots, \rho_b} (\rho)
      \right\|_{\tmop{TV}} & \leqslant & \int \left| \rho \right| (\mathd
      \sigma_0) \prod_{i = 1}^b \left\| \rho_i^{\sigma_0} \right\|_{\tmop{TV}}
      .  \label{R:norm}
    \end{eqnarray}
  \end{enumerate}
\end{proposition}

\begin{proof}
  
  \begin{enumerate}
    \item We write $\rho = \rho^+ - \rho^-$ the Hahn-Jordan decomposition of
    $\rho$ into positive, mutually singular measures. We have
    \begin{eqnarray*}
      \left\| R^{\tmmathbf{\eta}}_{\rho_1, \ldots, \rho_b} \left( \rho \right)
      \right\|_{\tmop{TV}} & \leqslant & \left\| R^{\tmmathbf{\eta}}_{\rho_1,
      \ldots, \rho_b} \left( \rho^+ \right) \right\|_{\tmop{TV}} + \left\|
      R^{\tmmathbf{\eta}}_{\rho_1, \ldots, \rho_b} \left( \rho^- \right)
      \right\|_{\tmop{TV}}
    \end{eqnarray*}
    But $R^{\tmmathbf{\eta}}_{\rho_1, \ldots, \rho_b} \left( \rho^+ \right)$
    is obviously a positive measure, and its mass is
    \begin{eqnarray*}
      R^{\tmmathbf{\eta}}_{\rho_1, \ldots, \rho_b} \left( \rho^+ \right)
      (\Sigma) & = & \int \rho^+ (\mathd \tmmathbf{\sigma}_0)
      \rho_1^{\tmmathbf{\sigma}_0} (\mathd \tmmathbf{\sigma}_1) \cdots
      \rho_b^{\tmmathbf{\sigma}_0} (\mathd \tmmathbf{\sigma}_b)\\
      & = & \rho^+ (\Sigma)
    \end{eqnarray*}
    therefore $\left\| R^{\tmmathbf{\eta}}_{\rho_1, \ldots, \rho_b} \left(
    \rho \right) \right\|_{\tmop{TV}} \leqslant \left\| \rho
    \right\|_{\tmop{TV}}$ . Now we assume that $\rho \in \mathcal{M}_0$, that
    is, $\rho^+ (\Sigma) = \rho^- (\Sigma)$. The same calculation as above
    shows that $R^{\tmmathbf{\eta}}_{\rho_1, \ldots, \rho_b} \left( \rho
    \right) \in \mathcal{M}_0$. Now, given $\tmmathbf{\sigma}_0,
    \tmmathbf{\eta} \in \Sigma$ we consider the probability measure
    \begin{eqnarray*}
      \varphi^{\tmmathbf{\eta}, \tmmathbf{\sigma}_0} & = & \int
      \rho_1^{\tmmathbf{\sigma}_0} (\mathd \tmmathbf{\sigma}_1) \cdots
      \rho_b^{\tmmathbf{\sigma}_0} (\mathd \tmmathbf{\sigma}_b) \mu_{\beta,
      \tmmathbf{h}+\tmmathbf{\eta}+\tmmathbf{\sigma}_1 + \cdots
      +\tmmathbf{\sigma}_b, \lambda} .
    \end{eqnarray*}
    It is immediate from (\ref{mu:gamma}) that
    \begin{eqnarray*}
      | \varphi^{\tmmathbf{\eta}, \tmmathbf{\sigma}_0} (A) -
      \varphi^{\tmmathbf{\eta}, \tmmathbf{\sigma}'_0} (A) | & \leqslant &
      \gamma \text{,} \,\,\,\, \forall \tmmathbf{\eta}, \tmmathbf{\sigma}_0,
      \tmmathbf{\sigma}_0' \in \Sigma, \forall A \in \mathcal{B} (\Sigma) .
    \end{eqnarray*}
    Consequently,
    \begin{gather*}
      R^{\tmmathbf{\eta}}_{\rho_1, \ldots, \rho_b} \left( \rho \right) (A) -
      R^{\tmmathbf{\eta}}_{\rho_1, \ldots, \rho_b} \left( \rho \right) (A^c) 
      \\=  \int \rho^+ (\mathd \tmmathbf{\sigma}_0) \left[
      \varphi^{\tmmathbf{\eta}, \tmmathbf{\sigma}_0} (A) -
      \varphi^{\tmmathbf{\eta}, \tmmathbf{\sigma}_0} (A^c) \right]
       - \int \rho^- (\mathd \tmmathbf{\sigma}_0) \left[
      \varphi^{\tmmathbf{\eta}, \tmmathbf{\sigma}_0} (A) -
      \varphi^{\tmmathbf{\eta}, \tmmathbf{\sigma}_0} (A^c) \right]\\
      =  2 \int \frac{\rho^+ (\mathd \tmmathbf{\sigma}_0) \rho^- (\mathd
      \tmmathbf{\sigma}_0')}{\rho^+ \left( \Sigma \right)} \left[
      \varphi^{\tmmathbf{\eta}, \tmmathbf{\sigma}_0} (A) -
      \varphi^{\tmmathbf{\eta}, \tmmathbf{\sigma}_0'} (A) \right]\\
       \leqslant  2 \gamma \| \rho \|_{\tmop{TV}}
    \end{gather*}
    since $\rho^+ (\Sigma) = \rho^- (\Sigma) = \| \rho \|_{\tmop{TV}}$. This
    shows that $R^{\tmmathbf{\eta}}_{\rho_1, \ldots, \rho_b}$ is a
    $\gamma$-contraction on $\mathcal{M}_0$.
    
    \item It is an immediate consequence of the first point that $\rho \in
    \mathcal{M}_0 \mapsto \rho - R^{\tmmathbf{\eta}}_{\rho_1, \ldots, \rho_b}
    \left( \rho \right) \in \mathcal{M}_0$ is invertible, with the given
    inverse. The computation of the operator norm of the inverse is immediate.
    
    \item This follows at once from the definition of total variation distance
    : \[\| \mu \|_{\tmop{TV}} = \sup_{f : |f| \leqslant 1} \mu (f) / 2.\]
  \end{enumerate}
\end{proof}

\subsubsection{Monotonicity of the resampling operator}

In this paragraph we examine the monotonicity properties of the resampling
operator. We recall that, when $\mu, \nu \in \mathcal{M}_{1, +}$, we say that
$\mu$ is stochastically larger than $\nu$ when, for all $f : \Sigma
\rightarrow {\mathbf R}$ increasing, $\mu \left( f \right) \geqslant \nu
\left( f \right)$. We generalize this notion by saying that $\rho \in
\mathcal{M}_0$ is {\tmem{stochastically positive}} when $\rho (f) \geqslant
0$, for all $f$ increasing, which we write $\rho \succ 0$. Note that this has
nothing to do with ``$\rho$ is a positive measure'', which itself is the same
as $\rho \in \mathcal{M}_+$.

\begin{remark}
  Let $\mu, \nu$ two probability measures and assume that $\mu$ is
  stochastically larger than $\nu$. Let $\rho = \mu - \nu$ and call
  $\rho^{\pm}$ the positive and negative parts of $\mu$. We claim that a
  coupling for the positive and negative parts of $\rho$ yields easily a
  monotone coupling of $\mu$ and $\nu$, which is optimal in the sense that it
  realizes total variation distance. In other words, there always exists an
  optimal and monotone coupling between two ordered probability measures. An
  explicit construction of this coupling is given below. By definition of the
  total variation distance,
  \begin{eqnarray*}
    \rho^+ (\Sigma) = \rho^- (\Sigma) & = & \| \mu - \nu \|_{\tmop{TV}} .
  \end{eqnarray*}
  Assume that $\mu \neq \nu$ and consider $\varphi$ a monotone coupling of the
  stochastically ordered probability measures $\| \mu - \nu \|_{\tmop{TV}}^{-
  1} \rho^{\pm}$. Note that $(\mu - \rho^+) = (\nu - \rho^-)$ is a positive
  measure with mass $1 - \| \mu - \nu \|_{\tmop{TV}}$. We call $\psi$ the law
  of $(\sigma, \sigma)$ when $\sigma \sim (\mu - \rho^+) / (1 -\| \mu - \nu
  \|_{\tmop{TV}})$. Then,
  \begin{eqnarray*}
    \| \mu - \nu \|_{\tmop{TV}} \varphi + \left( 1 -\| \mu - \nu
    \|_{\tmop{TV}} \right) \psi &  & 
  \end{eqnarray*}
  is a monotone and optimal coupling for $(\mu, \nu)$.
\end{remark}

We begin by observing that $R^{\tmmathbf{\eta}}_{\rho_1, \ldots, \rho_b}
\left( \rho \right)$ is increasing in all parameters $\rho \in \mathcal{M}_{+,
1}$, $\rho_1, \ldots, \rho_b \in \mathcal{X}_{+, 1}$ and $\tmmathbf{\eta}$
such that for all $i \in \{1, \ldots, b\}$, $\rho_i^{\tmmathbf{\sigma}_0}$
increases stochastically with the boundary condition $\tmmathbf{\sigma}_0 \in
\Sigma$. 
More precisely, $R^{\bar{\tmmathbf{\eta}}}_{\bar{\rho}_1, \ldots,
\bar{\rho}_b} \left( \rho \right) \succ R^{\tmmathbf{\eta}}_{\rho_1,
\ldots, \rho_b} \left( \rho \right)$ \ if $\rho \in
\mathcal{M}_{1, +}$, $\rho_1, \ldots, \rho_b, \bar{\rho}_1, \ldots,
\bar{\rho}_b \in \mathcal{X}_{+, 1}$ and $\tmmathbf{\eta},
\bar{\tmmathbf{\eta}} \in L^1 \left( \left[ 0, \beta \right] \right)$ satisfy
$\bar{\tmmathbf{\eta}} \geqslant \tmmathbf{\eta}$ and
$\bar{\rho}_i^{\tmmathbf{\sigma}_0} \succ \rho_i^{\tmmathbf{\sigma}_0}$, $i =
1 \ldots b$, for any $\tmmathbf{\sigma}_0 \in \Sigma$. Also
$R^{\tmmathbf{\eta}}_{\rho_1,
\ldots, \rho_b} \left( \bar{\rho} \right)
\succ R^{\tmmathbf{\eta}}_{\rho_1,
\ldots, \rho_b} \left( \rho \right)$ \ if $\rho, \bar{\rho} \in
\mathcal{M}_{1, +}$, $\rho_1, \ldots, \rho_b \in \mathcal{X}_{+, 1}$ and $\tmmathbf{\eta}
 \in L^1 \left( \left[ 0, \beta \right] \right)$ are such that
$\bar{\rho} \succ \rho$ and 
$\rho_i^{\tmmathbf{\sigma}_0}$
increases stochastically with  $\tmmathbf{\sigma}_0 \in
\Sigma$, $i =
1 \ldots b$.
Both statements follow
immediately from equation (\ref{R}) and Theorem \ref{thm:mono}.

It appends that one can compare the stochastic increase in the parameter $\rho_1$ with the norm of the difference:

\begin{proposition}
  \label{prop:Sinc}Let $\tmmathbf{\eta} \in \Sigma$. Let $c > 0$ and $\Gamma$
  be the constants in Proposition~\ref{prop:gamma} and Theorem~\ref{thm:mono}
  corresponding to $\beta$ and $M = \max \left( \|\tmmathbf{h}\|_1 + \beta (b
  + 1), \lambda \right)$. Let $f : \Sigma \rightarrow {\mathbf R}$ be
  increasing. Assume that $\rho \in \mathcal{M}_{+, 1}$, $\rho_1 \in
  \mathcal{X}_0$ with $\rho_1^{\tmmathbf{\sigma}_0} \succ 0$ for all
  $\tmmathbf{\sigma}_0 \in \Sigma$ and let $\rho_2, \ldots, \rho_b \in
  \mathcal{X}_{+, 1}$. Then
  \begin{enumerate}
    \item
    \begin{eqnarray*}
      \left\| R^{\tmmathbf{\eta}}_{\rho_1, \ldots, \rho_b} (\rho)
      \right\|_{\tmop{TV}} & \leqslant & \Gamma \int \rho (\mathd
      \tmmathbf{\sigma}_0) \rho_1^{\tmmathbf{\sigma}_0} (\tmmathbf{\sigma}
      \bullet \tmmathbf{1}) .
    \end{eqnarray*}
    \item
    \begin{eqnarray*}
      R^{\tmmathbf{\eta}}_{\rho_1, \ldots, \rho_b} (\rho) (f) & \geqslant & c
      (f ( +) - f (-)) \int \rho (\mathd \tmmathbf{\sigma}_0)
      \rho^{\tmmathbf{\sigma}_0}_1 (\tmmathbf{\sigma} \bullet \tmmathbf{1}) .
    \end{eqnarray*}
  \end{enumerate}
\end{proposition}

\begin{proof}
  
  \begin{enumerate}
    \item The assumption that $\rho_1^{\tmmathbf{\sigma}_0}$ is stochastically
    positive means exactly that its positive part is stochastically larger
    than its negative part. We consider therefore $\Psi^{\sigma_0}$ a monotone
    coupling of these two parts:
    \begin{gather*}
      \left\| R^{\tmmathbf{\eta}}_{\rho_1, \ldots, \rho_b} (\rho)
      \right\|_{\tmop{TV}}  \\=  \frac{1}{2} \sup_{f : |f| \leqslant 1} \int
      \rho (\mathd \tmmathbf{\sigma}_0) \mathd \Psi^{\tmmathbf{\sigma}_0}
      (\tmmathbf{\sigma}_1^+, \tmmathbf{\sigma}_1^-)
      \rho_2^{\tmmathbf{\sigma}_0} (\mathd \tmmathbf{\sigma}_2) \cdots
      \rho_b^{\tmmathbf{\sigma}_0} (\mathd \tmmathbf{\sigma}_b) \left[
      \mu_{\tmmathbf{h}+\tmmathbf{\eta}+\tmmathbf{\sigma}^+_1 + \cdots
      +\tmmathbf{\sigma}_b} \left( f \right) - \mu_{
      \tmmathbf{h}+\tmmathbf{\eta}+\tmmathbf{\sigma}^-_1 + \cdots
      +\tmmathbf{\sigma}_b} \left( f \right) \right]\\
      \leqslant  \Gamma \int \rho (\mathd \tmmathbf{\sigma}_0) \mathd
      \Psi^{\tmmathbf{\sigma}_0} (\tmmathbf{\sigma}_1^+,
      \tmmathbf{\sigma}_1^-)\|\tmmathbf{\sigma}_1^+ -\tmmathbf{\sigma}_1^-
      \|_1\\
       =  \Gamma \int \rho (\mathd \tmmathbf{\sigma}_0) \mathd
      \Psi^{\tmmathbf{\sigma}_0} (\tmmathbf{\sigma}_1^+,
      \tmmathbf{\sigma}_1^-) (\tmmathbf{\sigma}_1^+ -\tmmathbf{\sigma}_1^-)
      \bullet \tmmathbf{1}\\
      =  \Gamma \int \rho (\mathd \tmmathbf{\sigma}_0)
      \rho^{\tmmathbf{\sigma}_0}_1 (\tmmathbf{\sigma} \bullet \tmmathbf{1})
    \end{gather*}
    \item Let $F (\tmmathbf{\sigma}_1) = \int \rho_2^{\tmmathbf{\sigma}_0}
    (\mathd \tmmathbf{\sigma}_2) \cdots \rho_b^{\tmmathbf{\sigma}_0} (\mathd
    \tmmathbf{\sigma}_b) \mu_{
    \tmmathbf{h}+\tmmathbf{\eta}+\tmmathbf{\sigma}_1 + \cdots
    +\tmmathbf{\sigma}_b} \left( f \right)$, then
    $\tmmathbf{\sigma}_1 \mapsto F \left( \tmmathbf{\sigma}_1 \right)$ is
    non-decreasing and moreover, for every $\tmmathbf{\sigma}_1^- \leqslant
    \tmmathbf{\sigma}_1^+$,
    \begin{eqnarray*}
      F (\tmmathbf{\sigma}_1^+) - F (\tmmathbf{\sigma}_1^-) & \geqslant & c (f
      ( +) - f (-))  (\tmmathbf{\sigma}_1^+ -\tmmathbf{\sigma}_1^-) \bullet
      \tmmathbf{1}
    \end{eqnarray*}
    according to (\ref{mu:stpos}) in Theorem~\ref{thm:mono}. When we introduce
    the same coupling of the positive and negative parts of
    $\rho_1^{\tmmathbf{\sigma}_0}$ as in the former paragraph, we obtain the
    sought lower bound.
  \end{enumerate}
\end{proof}

\subsubsection{The cavity equation has a unique solution}

\begin{theorem}
  \label{thm:phi}Let $\rho_1, \ldots, \rho_b \in \mathcal{X}_{+, 1}$ and
  $\tmmathbf{\eta} \in L^1 \left( \left[ 0, \beta \right] \right)$. The cavity
  equation (\ref{cav}) has a unique solution in $\mathcal{M}_{+, 1}$, that we
  call $\Phi^{\tmmathbf{\eta}}_{\rho_1, \ldots, \rho_b}$. Furthermore, this
  solution satisfies:
  \begin{eqnarray}
    \Phi^{\tmmathbf{\eta}}_{\bar{\rho}_1, \rho_2, \ldots, \rho_b} -
    \Phi^{\tmmathbf{\eta}}_{\rho_1, \ldots, \rho_b} & = & \left( I -
    R^{\tmmathbf{\eta}}_{\bar{\rho}_1, \rho_2, \ldots, \rho_b} \right)^{- 1}
    R^{\tmmathbf{\eta}}_{\bar{\rho}_1 - \rho_1, \rho_2, \ldots, \rho_b} \left(
    \Phi^{\tmmathbf{\eta}}_{\rho_1, \ldots, \rho_b} \right) \text{, \ }
    \forall \bar{\rho}_1 \in \mathcal{X}_{1, +} .  \label{Phib}
  \end{eqnarray}
  Moreover, the solution to the cavity equation
  $\Phi^{\tmmathbf{\eta}}_{\rho_1, \ldots, \rho_b}$ increases stochastically
  with its parameters $\rho_1, \ldots, \rho_b \in \mathcal{X}_{+, 1}$ increasing with the boundary condition, and  $\tmmathbf{\eta} \in L^1 \left( \left[ 0, \beta \right] \right)$. More
  precisely, $\Phi^{\bar{\tmmathbf{\eta}}}_{\bar{\rho}_1, \ldots,
  \bar{\rho}_b} \succ \Phi^{\tmmathbf{\eta}}_{\rho_1, \ldots, \rho_b}$ \ if
  $\rho_1, \ldots, \rho_b, \bar{\rho}_1, \ldots, \bar{\rho}_b \in
  \mathcal{X}_{+, 1}$ and $\tmmathbf{\eta}, \bar{\tmmathbf{\eta}} \in L^1
  \left( \left[ 0, \beta \right] \right)$ satisfy $\bar{\tmmathbf{\eta}}
  \geqslant \tmmathbf{\eta}$ and $\bar{\rho}_i^{\tmmathbf{\sigma}_0} \succ
  \rho_i^{\tmmathbf{\sigma}_0}$, $i = 1 \ldots b$, for any
  $\tmmathbf{\sigma}_0 \in \Sigma$, with both $\bar{\rho}_i^{\tmmathbf{\sigma}_0}$ and $ \rho_i^{\tmmathbf{\sigma}_0}$ being stochastically increasing in $\tmmathbf{\sigma}_0 \in \Sigma$.
\end{theorem}

\begin{proof}
  We consider first the case when $\rho_i \in \mathcal{X}_{+, 1}$ are
  constant, i.e.~$\rho_i^{\tmmathbf{\sigma}}$ does not depend on
  $\tmmathbf{\sigma} \in \Sigma$, for all $i = 1 \ldots b$. Then,
  $R^{\tmmathbf{\eta}}_{\rho_1, \ldots, \rho_b} \left( \rho \right)$ depends
  uniquely on the mass of $\rho \in \mathcal{M}$. Consequently the unique
  solution to the cavity equation in $\mathcal{M}_{+, 1}$, for this
  parameters, is $R^{\tmmathbf{\eta}}_{\rho_1, \ldots, \rho_b} \left( \nu
  \right)$ where $\nu$ is an arbitrary probability measure on $\Sigma$. Now we
  consider $\rho_1, \ldots, \rho_b \in \mathcal{X}_{+, 1}$ such that the
  cavity equation has a unique solution $\Phi^{\tmmathbf{\eta}}_{\rho_1,
  \ldots, \rho_b}$ together with an arbitrary $\bar{\rho}_1 \in
  \mathcal{X}_{+, 1}$ and $\rho \in \mathcal{M}_{+, 1}$. We use the shorter
  notations $R = R^{\tmmathbf{\eta}}_{\rho_1, \ldots, \rho_b}$, $\bar{R} =
  R^{\tmmathbf{\eta}}_{\bar{\rho}_1, \ldots, \rho_b}$ and $\Phi =
  \Phi^{\tmmathbf{\eta}}_{\rho_1, \ldots, \rho_b}$. We have
  \begin{eqnarray*}
    \rho = \bar{R} \left( \rho \right) & \Leftrightarrow & \rho - \Phi =
    \bar{R} \left( \rho \right) - R \left( \Phi \right)\\
    & \Leftrightarrow & \rho - \Phi = \bar{R} \left( \rho - \Phi \right) +
    \bar{R} \left( \Phi \right) - R \left( \Phi \right)\\
    & \Leftrightarrow & \left( I - \bar{R} \right) \left( \rho - \Phi \right)
    = \left( \bar{R} - R \right) \left( \Phi \right) .
  \end{eqnarray*}
  But $\rho - \Phi \in \mathcal{M}_0$, where $I - \bar{R}$ is invertible
  because of Proposition \ref{prop:Scont}. Consequently there is a unique
  solution to the cavity equation with parameters $\bar{\rho}_1, \rho_2,
  \ldots, \rho_b \in \mathcal{X}_{+, 1}$ and it is given by formula
  (\ref{Phib}). In the same way, when we change any other of the $\rho_i \in
  \mathcal{X}_{+, 1}$ the solution to the cavity equation still exists and
  remains unique, so we can change all $\rho_1, \ldots, \rho_b$ to arbitrary
  elements of $\mathcal{X}_{+, 1}$.
  
  The monotonicity of the solution to the cavity equation follows at once from
  the fact that, for any probability measure $\nu \in \mathcal{M}_{1, +}$,
  \begin{eqnarray}
    \Phi^{\tmmathbf{\eta}}_{\rho_1, \ldots, \rho_b} & = & \lim_k \left(
    R^{\tmmathbf{\eta}}_{\rho_1, \ldots, \rho_b} \right)^k \left( \nu \right) 
    \label{Phi:Rk}
  \end{eqnarray}
  (according to the first point of Proposition \ref{prop:Scont}) together with
  the monotonicity of $R^{\tmmathbf{\eta}}_{\rho_1, \ldots, \rho_b} \left(
  \rho \right)$ in all parameters $\rho \in \mathcal{M}_{+, 1}$, $\rho_1,
  \ldots, \rho_b \in \mathcal{X}_{+, 1}$ and $\tmmathbf{\eta}$ such that for
  all $i \in \{1, \ldots, b\}$, $\rho_i^{\tmmathbf{\sigma}_0}$
increases stochastically with the boundary condition $\tmmathbf{\sigma}_0 \in
\Sigma$. 
\end{proof}

\subsubsection{Existence of a fixed point for the cavity equation}

We consider
\begin{eqnarray}
  \nu^{\tmmathbf{\eta}}_n & = & \mu_{\mathbbm{T}_n}^{\tmmathbf{\eta}, +}
  (\tmmathbf{\sigma}_r \in .)  \label{nu:def}
\end{eqnarray}
for each $n \geqslant 0$ and $\tmmathbf{\eta} \in \Sigma$, where
$\mu_{\mathbbm{T}_n}^{\tmmathbf{\eta}, +}$ is the quantum Gibbs measure with
parameters $\beta, \tmmathbf{h}, \lambda$ on the tree $\mathbbm{T}_n$, with
plus boundary condition on the leaves and field $\tmmathbf{\eta}$ at the root.
Observe that $\nu_0 = \delta_+$ and $\nu_{n + 1} = \Phi_{\nu_n, \ldots,
\nu_n}$, the solution to the cavity equation with parameters $\nu_n, \ldots,
\nu_n$.

\begin{lemma}
  Let $\beta, \lambda, M > 0$. There is $c > 0$ such that, for all
  $\tmmathbf{h}, \tmmathbf{\eta}, \bar{\tmmathbf{\eta}}$ with $L^1$-norm at
  most $M$ and all $n \in \mathbbm{N}^{\star}$,
  \begin{enumerate}
    \item If $\bar{\tmmathbf{\eta}} \geqslant \tmmathbf{\eta}$,
    \begin{eqnarray}
      (\nu^{\bar{\tmmathbf{\eta}}}_n - \nu^{\tmmathbf{\eta}}_n) (\sigma
      \bullet \tmmathbf{1}) & \geqslant & c ( \bar{\tmmathbf{\eta}}
      -\tmmathbf{\eta}) \bullet \tmmathbf{1}.  \label{nu:pos}
    \end{eqnarray}
    \item
    \begin{equation}
      c \leqslant \frac{\mathd \nu_n^{\tmmathbf{\eta}}}{\mathd \varphi}
      \leqslant c^{- 1} \label{dnu:bounded}
    \end{equation}
  \end{enumerate}
\end{lemma}

\begin{proof}
  The first point is a consequence of Theorem \ref{thm:mono} together with the
  DLR equation while the second one follows from the definition of
  $\mu_{\mathbbm{T}_n}^{\tmmathbf{\eta}, +}$ together with the DLR equation.
\end{proof}

\begin{proposition}
  For every $\tmmathbf{\eta} \in L^1 \left( \left[ 0, \beta \right] \right)$,
  the stochastically decreasing sequence of probability distributions
  $\nu^{\tmmathbf{\eta}}_n$ converge to a probability distribution
  $\nu_{\infty}^{\tmmathbf{\eta}}$ on $\Sigma$. The conditional distribution
  $\nu_{\infty} \in \mathcal{X}_{1, +}$ is a fixed point of the cavity
  equation, i.e.~$\nu_{\infty} = \Phi_{\nu_{\infty}, \ldots, \nu_{\infty}}$.
\end{proposition}

\begin{proof}
  The fact that the sequence $\nu^{\tmmathbf{\eta}}_n$ is stochastically
  decreasing is an inductive consequence of the monotonicity of the solution to the cavity equation as $\nu_1
  \prec \nu_0 = \delta_+$. Now consider the $\pi$-system
  \begin{eqnarray*}
    \Pi & = & \left\{ \left\{ \sigma \left( t \right) = +, \forall t \in I
    \right\}, I \subset \left[ 0, \beta \right] \text{ finite} \right\} .
  \end{eqnarray*}
  Obviously $\sigma \left( \Pi \right)$ is the whole $\sigma$-algebra
  corresponding to the Skorohod topology. Furthermore, any $A \in \Pi$ is an
  increasing event and consequently $\nu^{\tmmathbf{\eta}}_n \left( A \right)$
  has a decreasing limit. On the other hand, there is $C < \infty$ such that
  the Radon-Nikodym derivative $\frac{\mathd \nu^{\tmmathbf{\eta}}_n}{\mathd
  \varphi}$ is uniformly bounded by $C$, for any $n \geqslant 0$, and
  consequently $\left( \nu_n^{\tmmathbf{\eta}} \right)$ is tight. Indeed, if
  $\Sigma_k$ is the compact set of spin configurations with at most $k$ flips,
  then $\nu_n^{\tmmathbf{\eta}} \left( \Sigma_k^c \right) \leqslant C \varphi
  \left( \Sigma_k^c \right)$. This proves existence and uniqueness of the
  limit $\nu_{\infty}^{\tmmathbf{\eta}}$. The conclusion that $\nu_{\infty}$
  is a fixed point of the cavity equation is an obvious consequence of the
  continuity of the solution to the cavity equation along its parameters, see
  (\ref{Phib}) and (\ref{R:norm}). Note that if we are not in the uniqueness
  regime, the cavity equation has another fixed point corresponding to minus
  boundary condition.
\end{proof}

\subsubsection{Derivative of the solution of the cavity equation at the fixed
point}

Two natural norms on $\mathcal{X}$ are
\begin{eqnarray}
  \| \rho \|_{\infty, \mathcal{X}} & = & \sup_{\tmmathbf{\eta}} \|
  \rho^{\tmmathbf{\eta}} \|_{\tmop{TV}}  \label{Xinf}\\
  \| \rho \|_{1, \mathcal{X}} & = & \int \mathd \varphi (\tmmathbf{\eta})\|
  \rho^{\tmmathbf{\eta}} \|_{\tmop{TV}} .  \label{X1}
\end{eqnarray}
The first norm makes of $\mathcal{X}$ a Banach space as $\mathcal{M}$ itself
is a Banach space. As far as the second norm is concerned, we observe that
when $\rho^{\tmmathbf{\eta}}$ is absolutely continuous with respect to
$\varphi$, $\| \rho \|_{1, \mathcal{X}} = \frac{1}{2} \int \mathd \varphi
(\tmmathbf{\eta}) \mathd \varphi \left( \sigma \right) \left| \frac{\mathd
\rho^{\tmmathbf{\eta}}}{\mathd \varphi} \right|$.

Now we introduce the derivative of the solution of the cavity equation
$\Phi^{\tmmathbf{\eta}}_{\rho_1, \ldots, \rho_b}$ along the first variable
$\rho_1$ in the direction $\rho \in \mathcal{X}_0$, at the fixed point $\rho_1
= \cdots = \rho_b = \nu_{\infty}$. The formula follows at once from
(\ref{Phib}):
\begin{eqnarray}
  D^{\tmmathbf{\eta}} (\rho) & = & \left( I -
  R^{\tmmathbf{\eta}}_{\nu_{\infty}, \ldots, \nu_{\infty}} \right)^{- 1}
  R^{\tmmathbf{\eta}}_{\rho, \nu_{\infty}, \ldots, \nu_{\infty}} \left(
  \nu^{\tmmathbf{\eta}}_{\infty} \right) \text{,} \forall \rho \in
  \mathcal{X}_0 .  \label{D}
\end{eqnarray}
Note that, for all $\tmmathbf{\eta} \in \Sigma$ and $\rho \in \mathcal{X}_0$,
$D^{\tmmathbf{\eta}} \left( \rho \right) \in \mathcal{M}_0$. Consequently $D
\left( \rho \right) : \tmmathbf{\eta} \mapsto D^{\tmmathbf{\eta}} \left( \rho
\right) \in \mathcal{X}_0$ and $D$ can be seen as a linear operator on
$\mathcal{X}_0$. Our main theorem gives a bound on the norm of the operator
norm of the $k$-th iterate of $D$, denoted by $D^k$, from the space
$(\mathcal{X}_0, \|.\|_{1, \mathcal{X}})$ to $(\mathcal{X}_0, \|.\|_{\infty,
\mathcal{X}})$:

\begin{theorem}
  \label{thm:Dkrho}There is $C < \infty$ that depends only on $\beta,
  \tmmathbf{h}, \lambda, b$ such that, for all $\rho \in \mathcal{X}_0$ and
  for all $k \geqslant 1$,
  \begin{eqnarray}
    \|D^k (\rho)\|_{\infty, \mathcal{X}} & \leqslant & \frac{C}{b^k}  \| \rho
    \|_{1, \mathcal{X}} .  \label{Dk:norm}
  \end{eqnarray}
  In particular, the spectral radius of $D$ is at most $1 / b$.
\end{theorem}

\begin{remark}
The Krein-Rutman theorem {\cite{KR50}} states that the spectral radius of a strictly positive and compact operator is an eigenvalue corresponding to a positive eigenvector. Although we do not use this theorem, it is remarkable that in the proof below the asymptotic direction of the stochastically positive conditional measure $\nu_n-\nu_\infty$ helps to control the norm of $D$.\end{remark}

\begin{proof}
  We first establish two preliminary results. Let $\rho \in \mathcal{X}_0$ and
  assume that $\rho^{\tmmathbf{\sigma}_0}$ is stochastically positive, for
  every $\tmmathbf{\sigma}_0 \in \Sigma$. Then ({\tmem{i}}) in Proposition
  \ref{prop:Sinc} together with ({\tmem{ii}}) in Proposition \ref{prop:Scont}
  imply that, for any $\tmmathbf{\eta} \in \Sigma$,
  \begin{eqnarray}
    \|D^{\tmmathbf{\eta}} (\rho)\|_{\tmop{TV}} & \leqslant & \frac{\Gamma}{1 -
    \gamma} \int \nu_{\infty}^{\tmmathbf{\eta}} (\mathd \tmmathbf{\sigma}_0)
    \rho^{\tmmathbf{\sigma}_0} (\tmmathbf{\sigma} \bullet \tmmathbf{1})
    \nonumber\\
    & \leqslant & \frac{C \Gamma}{1 - \gamma} \int \varphi (\mathd
    \tmmathbf{\sigma}_0) \rho^{\tmmathbf{\sigma}_0} (\tmmathbf{\sigma} \bullet
    \tmmathbf{1}) .  \label{Dtv}
  \end{eqnarray}
  where in the second inequality we have used (\ref{dnu:bounded}). On the
  other hand, with the same assumption on $\rho$, if we take $f : \Sigma
  \rightarrow {\mathbf R}$ increasing with $f (+) = - f (-) = 1$, then
  \begin{eqnarray}
    D^{\tmmathbf{\eta}} (\rho) (f) & \geqslant & c \int \varphi (\mathd
    \tmmathbf{\sigma}_0) \rho^{\tmmathbf{\sigma}_0} (\tmmathbf{\sigma} \bullet
    \tmmathbf{1})  \label{Df}
  \end{eqnarray}
  for some $c > 0$ depending only on $\beta, \tmmathbf{h}, \lambda, b$.
  Indeed, each term in the expansion of $\left( I -
  R^{\tmmathbf{\eta}}_{\nu_{\infty}, \ldots, \nu_{\infty}} \right)^{- 1}$
  leaves the set of stochastically positive $\rho \in \mathcal{X}_0$
  invariant. So
  \begin{eqnarray*}
    D^{\tmmathbf{\eta}} (\rho) (f) & \geqslant & R^{\tmmathbf{\eta}}_{\rho,
    \nu_{\infty}, \ldots, \nu_{\infty}} \left( \nu_{\infty}^{\tmmathbf{\eta}}
    \right) \left( f \right)\\
    & \geqslant & 2 c \int \nu_{\infty}^{\tmmathbf{\eta}} (\mathd
    \tmmathbf{\sigma}_0) \rho^{\tmmathbf{\sigma}_0} (\tmmathbf{\sigma} \bullet
    \tmmathbf{1})
  \end{eqnarray*}
  where at the second line we use ({\tmem{ii}}) in Proposition
  \ref{prop:Sinc}. Inequality (\ref{Df}) follows then from
  (\ref{dnu:bounded}). Note that already inequalities (\ref{Dtv}) and
  (\ref{Df}) show that $\|D (\rho)\|_{\infty, \mathcal{X}} \leqslant C \|D
  (\rho)\|_{1, \mathcal{X}}$ when $\rho^{\tmmathbf{\sigma}_0}$ is
  stochastically positive, for every $\tmmathbf{\sigma}_0 \in \Sigma$.
  
  The derivative $D$ is closely related to the way $\nu_n$ converges to its
  limit $\nu_{\infty}$. As a consequence of (\ref{Phib}) in
  Theorem~\ref{thm:phi},
  \begin{eqnarray}
    \nu^{\tmmathbf{\eta}}_{n + 1} - \nu^{\tmmathbf{\eta}}_{\infty} & = &
    \Phi^{\tmmathbf{\eta}}_{\nu_n, \ldots, \nu_n} - \Phi_{\nu_{\infty},
    \ldots, \nu_{\infty}}^{\tmmathbf{\eta}} \nonumber\\
    & = & \Phi^{\tmmathbf{\eta}}_{\nu_n, \ldots, \nu_n} -
    \Phi^{\tmmathbf{\eta}}_{\nu_n, \ldots, \nu_n, \nu_{\infty}} + \nonumber\\
    &  & \Phi^{\tmmathbf{\eta}}_{\nu_n, \ldots, \nu_n, \nu_{\infty}} -
    \Phi^{\tmmathbf{\eta}}_{\nu_n, \ldots, \nu_n, \nu_{\infty}, \nu_{\infty}}
    + \nonumber\\
    &  & \vdots \nonumber\\
    &  & \Phi^{\tmmathbf{\eta}}_{\nu_n, \nu_{\infty}, \ldots, \nu_{\infty}} -
    \Phi_{\nu_{\infty}, \ldots, \nu_{\infty}}^{\tmmathbf{\eta}} \nonumber\\
    & = & \left( I - R^{\tmmathbf{\eta}}_{\nu_n, \ldots, \nu_n} \right)^{- 1}
    R^{\tmmathbf{\eta}}_{\nu_n - \nu_{\infty}, \nu_n, \ldots, \nu_n} \left(
    \Phi^{\tmmathbf{\eta}}_{\nu_n, \ldots, \nu_n, \nu_{\infty}} \right) +
    \nonumber\\
    &  & \vdots \nonumber\\
    &  & \left( I - R^{\tmmathbf{\eta}}_{\nu_n, \nu_{\infty}, \ldots,
    \nu_{\infty}} \right)^{- 1} R^{\tmmathbf{\eta}}_{\nu_n - \nu_{\infty},
    \nu_{\infty}, \ldots, \nu_{\infty}} \left(
    \Phi^{\tmmathbf{\eta}}_{\nu_{\infty}, \ldots, \nu_{\infty}, \nu_{\infty}}
    \right) \nonumber\\
    & = & bD^{\tmmathbf{\eta}} (\nu_n - \nu_{\infty}) + r_n^{\tmmathbf{\eta}}
    \label{bD}
  \end{eqnarray}
  with $\|r_n^{\tmmathbf{\eta}} \|_{\tmop{TV}} \leqslant C \| \nu_n -
  \nu_{\infty} \|_{\infty, \mathcal{X}}^2$ for some finite $C$, since
  $R^{\tmmathbf{\eta}}_{\rho_1, \ldots, \rho_b} (\rho)$ is multilinear and
  bounded, cf.~(\ref{R:norm}). It follows by induction that, for every $k
  \geqslant 1$,
  \begin{eqnarray}
    \nu_{n + k} - \nu_{\infty} & = & b^k D^k (\nu_n - \nu_{\infty}) + r_{k, n}
    \text{ as } n \rightarrow \infty,  \label{bkDk}
  \end{eqnarray}
  where $\|r_{k, n} \|_{\infty, \mathcal{X}} \leqslant C_k \| \nu_n -
  \nu_{\infty} \|_{\infty, \mathcal{X}}^2$.
  
  Now we consider $\rho \in \mathcal{X}_0$ such that
  $\rho^{\tmmathbf{\sigma}_0} \succ 0$ for every $\tmmathbf{\sigma}_0 \in
  \Sigma$. Without loss of generality we assume that $\left\| \rho
  \right\|_{1, \mathcal{X}} = 1$. From (\ref{bD}) we know that
  $\nu_n^{\tmmathbf{\eta}} - \nu^{\tmmathbf{\eta}}_{\infty}$ is close to the
  image by $b D^{\tmmathbf{\eta}}$ of $\nu_{n - 1}^{\tmmathbf{\eta}} -
  \nu^{\tmmathbf{\eta}}_{\infty}$ as $n \rightarrow \infty$. According to
  (\ref{Dtv}) and (\ref{Df}) this implies the existence of $c > 0$ such that,
  for all $n$ large enough, for every $f$ increasing with $f \left( \pm 1
  \right) = \pm 1$,
  \begin{eqnarray*}
    \inf_{\tmmathbf{\eta}} (\nu_n^{\tmmathbf{\eta}} -
    \nu^{\tmmathbf{\eta}}_{\infty}) (f) & \geqslant & c \sup_{\tmmathbf{\eta}}
    \left\| \nu_n^{\tmmathbf{\eta}} - \nu^{\tmmathbf{\eta}}_{\infty}
    \right\|_{\tmop{TV}} = c \left\| \nu_n - \nu_{\infty} \right\|_{\infty,
    \mathcal{X}} .
  \end{eqnarray*}
  From (\ref{Dtv}) and the assumption that
  $\left\| \rho  \right\|_{1, \mathcal{X}} = 1$
   it follows that
  \begin{eqnarray*}
    \sup_{\tmmathbf{\eta}} \|D^{\tmmathbf{\eta}} (\rho)\|_{\tmop{TV}} &
    \leqslant & C' = \frac{C \Gamma}{1-\gamma}
  \end{eqnarray*}
  According to the last two displays, for every $n$ large enough, for all
  $\tmmathbf{\eta} \in \Sigma$,
  \begin{eqnarray*}
    \frac{c}{2 C'} D^{\tmmathbf{\eta}} (\rho) & \prec &
    \frac{\nu_n^{\tmmathbf{\eta}} - \nu^{\tmmathbf{\eta}}_{\infty}}{\left\|
    \nu_n - \nu_{\infty} \right\|_{\infty, \mathcal{X}}} .
  \end{eqnarray*}
  Call $c' = c / \left(2 C' \right)$. As $D$ preserves stochastic positivity,
  \begin{eqnarray*}
    D^{k - 2} \left( \frac{\nu_n - \nu_{\infty}}{\left\| \nu_n - \nu_{\infty}
    \right\|_{\infty, \mathcal{X}}} - c' D (\rho) \right) & = & \frac{1}{b^{k
    - 2}}  \frac{\nu_{n + k - 2} - \nu_{\infty} - r_{k - 2, n}}{\left\| \nu_n
    - \nu_{\infty} \right\|_{\infty, \mathcal{X}}} - c' D^{k - 1} \left( \rho
    \right)
  \end{eqnarray*}
  is also stochastically positive ($r_{k - 2, n}$ was defined at (\ref{bD})).
  It follows that, for all $\tmmathbf{\eta} \in \Sigma$,
  \begin{eqnarray*}
    \left( D^{k - 1} \right)^{\tmmathbf{\eta}} \left( \rho \right)
    (\tmmathbf{\sigma} \bullet \tmmathbf{1}) & \leqslant & \frac{1}{c' b^{k -
    2}}  \frac{\left\| \nu_{n + k - 2} - \nu_{\infty} \right\|_{\infty,
    \mathcal{X}} + \left\| r_{k - 2, n} \right\|_{\infty,
    \mathcal{X}}}{\left\| \nu_n - \nu_{\infty} \right\|_{\infty, \mathcal{X}}}
    .
  \end{eqnarray*}
  for any $n$ large enough. We can obviously find a subsequence of $n$ along
  which $\left\| \nu_{n + k - 2} - \nu_{\infty} \right\|_{\infty, \mathcal{X}}
  \leqslant \left\| \nu_n - \nu_{\infty} \right\|_{\infty, \mathcal{X}}$, so
  taking $\liminf_n$ shows that
  \begin{eqnarray*}
    \left( D^{k - 1} \right)^{\tmmathbf{\eta}} \left( \rho \right)
    (\tmmathbf{\sigma} \bullet \tmmathbf{1}) & \leqslant & \frac{1}{c' b^{k -
    2}} .
  \end{eqnarray*}
  According to (\ref{Dtv}) we have
  \begin{eqnarray*}
    \left\| D^k \left( \rho \right) \right\|_{\infty, \mathcal{X}} & \leqslant
    & \sup_{\tmmathbf{\eta} \in \Sigma} \frac{C \Gamma}{1 - \gamma} \int
    \varphi (\mathd \tmmathbf{\sigma}_0) \left( D^{k - 1}
    \right)^{\tmmathbf{\sigma}_0} \left( \rho \right) (\tmmathbf{\sigma}
    \bullet \tmmathbf{1})
  \end{eqnarray*}
  and therefore, for some different $C < \infty$
  \begin{eqnarray}
    \left\| D^k \left( \rho \right) \right\|_{\infty, \mathcal{X}} & \leqslant
    & \frac{C}{b^k}  \label{Dk:norm:pos}
  \end{eqnarray}
  for all $k \geqslant 1$, all $\rho \in \mathcal{X}_0$ \ with $\left\| \rho
  \right\|_{1, \mathcal{X}} = 1$, such that $\rho^{\tmmathbf{\sigma}_0} \succ
  0$ for every $\tmmathbf{\sigma}_0 \in \Sigma$. Finally we extend
  (\ref{Dk:norm:pos}) to any $\rho \in \mathcal{X}_0$ with $\left\| \rho
  \right\|_{1, \mathcal{X}} = 1$. Consider
  \begin{eqnarray*}
    \rho_1^{\tmmathbf{\sigma}_0} & = & \left\| \rho^{\tmmathbf{\sigma}_0}
    \right\|_{\tmop{TV}} \left( \delta_+ - \delta_- \right) +
    \rho^{\tmmathbf{\sigma}_0}\\
    \rho_2^{\tmmathbf{\sigma}_0} & = & \left\| \rho^{\tmmathbf{\sigma}_0}
    \right\|_{\tmop{TV}} \left( \delta_+ - \delta_- \right) -
    \rho^{\tmmathbf{\sigma}_0} .
  \end{eqnarray*}
  Note that both $\rho_1^{\tmmathbf{\sigma}_0}$ and
  $\rho_2^{\tmmathbf{\sigma}_0}$ are stochastically positive for any
  $\tmmathbf{\sigma}_0 \in \Sigma$. For instance, if $f : \Sigma \rightarrow
  {\mathbf R}$ is increasing, then $\rho_1^{\tmmathbf{\sigma}_0} \left( f
  \right) = \left( f \left( + \right) - f \left( - \right) \right) \left\|
  \rho^{\tmmathbf{\sigma}_0} \right\|_{\tmop{TV}} + \rho^{\tmmathbf{\sigma}_0}
  \left( f \right) \geqslant 0$. According to the triangular inequality (note
  that $\left\| \rho_1 \right\|_{1, \mathcal{X}} \leqslant 2$) and to
  (\ref{Dk:norm:pos}),
  \begin{eqnarray*}
    \|D^k (\rho)\|_{\mathcal{X}}^{\infty} & \leqslant & \frac{4 C}{b^k}
  \end{eqnarray*}
  and we are done.
\end{proof}

\section{Proof of the main Theorems}

\subsection{Proof of Theorem \ref{thm:decay}}

Without loss of generality, we only consider here the case $x = 0$ in the
definition of $\kappa$. So we fix some $l \geqslant 0$ and $z \in
\mathbbm{T}_l^b$. We denote by $k = \left| z \right|$ the depth of $z$. Now we
consider the subtree of $\mathbbm{T}_l^b$ issued from $r$ that contains $z$,
and root it at $z$. This is a regular tree with $b + 1$ children at $z$ and
$b$ children otherwise (except on the leaves). Its depth is not uniform and
ranges between $\min \left( l - k, k \right)$ and $l + k$. The boundary
condition is uniformly plus, except on the leaf $r$ where it is $+$ or $-$.

Let $(z_0 = r, z_1, \ldots, z_k = z)$ be a path from $r$ to $z$. For any $i
\in \left\{ 1, \ldots, k - 1 \right\}$, we call $\nu_i^{\pm, \tmmathbf{\xi}}$
the conditional distribution $\mu \left( \tmmathbf{\sigma}_{z_i} \in \cdummy |
\tmmathbf{\sigma}_{z_{i + 1}} =\tmmathbf{\xi}, \tmmathbf{\sigma}_r = \pm
\right)$ and $\rho_{\left( i, j \right)}^{\tmmathbf{\xi}} = \mu \left(
\tmmathbf{\sigma}_{y_j} \in \cdummy | \tmmathbf{\sigma}_{z_{i + 1}}
=\tmmathbf{\xi} \right)$ where $y_j$ is the $j$-th children of $z_{i + 1}$ for
$j = 2, \ldots, b$, if we consider $z_i$ as its first children. We have,
\begin{eqnarray*}
  \nu_i^{\pm} & = & \tmmathbf{\Phi}_{\nu_{i - 1}^{\pm}, \rho_{\left( i - 1, 2
  \right)}, \ldots, \rho_{\left( i - 1, b \right)}} \text{, \ } i \leqslant k
  - 1
\end{eqnarray*}
and also, if we call $\nu_k^{\pm, 0}$ the marginal distribution $\mu \left(
\tmmathbf{\sigma}_z \in \cdummy | \tmmathbf{\sigma}_r = \pm \right)$ and
$\rho_j^{\tmmathbf{\xi}} = \mu \left( \tmmathbf{\sigma}_{y_j} \in \cdummy |
\tmmathbf{\sigma}_z =\tmmathbf{\xi} \right)$ where $y_j$ is the $j$-th
children of $z$ for $j = 2, \ldots, b + 1$, apart from $z_1$, then
\begin{eqnarray*}
  \nu_k^{\pm, 0} & = & \tmmathbf{\Phi}^0_{\nu^{\pm}_{k - 1}, \rho_2, \ldots,
  \rho_{b + 1}}
\end{eqnarray*}
(note the $b + 1$ parameters in the cavity equation).

Now we claim that, given a neighborhood $V \subset \mathcal{X}_{1, +}$ of
$\nu_{\infty}$ in the $\left\| . \right\|_{\infty, \mathcal{X}}$ norm, for all
$i$ such that $\min \left( i, l - i \right)$ is large enough depending on $V$,
every $\rho_{\left( i, j \right)}$ and $\nu_i^{\pm}$ lies in $V$. This is
clear for $\rho_{\left( i, j \right)}$ and $\nu_i^+$ as it is the marginal
distribution at the root of a large subtree with identical plus boundary
condition. As $\left\| \nu_{i + 1}^+ - \nu_{i + 1}^- \right\|_{\infty,
\mathcal{X}} \leqslant \gamma \left\| \nu_i^+ - \nu_i^- \right\|_{\infty,
\mathcal{X}}$ where $\gamma < 1$ as been defined in (\ref{gamma}), the same
holds for $\nu_i^-$.

According to (\ref{Phib}), for all $\varepsilon > 0$, for all $i$ with $\min
\left( i, l - i \right)$ large enough depending on $\varepsilon$ but not on
$k, l$,
\begin{eqnarray*}
  \left\| \nu_{i + 1}^+ - \nu_{i + 1}^- - D \left( \nu_i^+ - \nu_i^- \right)
  \right\|_{\infty, \mathcal{X}} & \leqslant & \varepsilon \left\| \nu_i^+ -
  \nu_i^- \right\|_{\infty, \mathcal{X}}
\end{eqnarray*}
where $D$, the derivative of the solution of the cavity equation, at the fixed
point $\nu_{\infty}$, along the first parameter, has been defined at
(\ref{D}). This clearly extends as follows. Fix $j \in \mathbbm{N}^{\star}$
and $\varepsilon > 0$. Then, there is $a = a \left( \varepsilon, j \right)$
that does not depend on $k, l$ such that, for all $i$ with $\min \left( i, l -
i \right) \geqslant a$, then
\begin{eqnarray*}
  \left\| \nu_{i + j}^+ - \nu_{i + j}^- - D^j \left( \nu_i^+ - \nu_i^- \right)
  \right\|_{\infty, \mathcal{X}} & \leqslant & \varepsilon \left\| \nu_i^+ -
  \nu_i^- \right\|_{\infty, \mathcal{X}} .
\end{eqnarray*}
According to Theorem \ref{thm:Dkrho} we conclude that, under the same
conditions,
\begin{eqnarray*}
  \left\| \nu_{i + j}^+ - \nu_{i + j}^- \right\|_{\infty, \mathcal{X}} &
  \leqslant & \left( \frac{C}{b^j} + \varepsilon \right) \left\| \nu_i^+ -
  \nu_i^- \right\|_{\infty, \mathcal{X}}
\end{eqnarray*}
where $C < \infty$ depends only on $\beta, \tmmathbf{h}, \lambda$. From
$\left\| \nu_{i + 1}^+ - \nu_{i + 1}^- \right\|_{\infty, \mathcal{X}}
\leqslant \gamma \left\| \nu_i^+ - \nu_i^- \right\|_{\infty, \mathcal{X}}$ it
follows that
\begin{eqnarray*}
  \left\| \nu_k^{+, 0} - \nu_k^{-, 0} \right\|_{\tmop{TV}} & \leqslant &
  \left( \frac{C}{b^j} + \varepsilon \right)^{\left[ \left( k - 2 a \right) /
  j \right]}
\end{eqnarray*}
which proves the first statement of the Theorem.

Now we consider $\beta, \tmmathbf{h}, \lambda$ in the uniqueness regime. Let
$\tmmathbf{\tau} \in \Sigma^{\mathbbm{T}_{\infty}^b}$ and denote
$\mu^{\tmmathbf{\tau}}_l$ the Gibbs measure on $\mathbbm{T}_l^b$ with
$\tmmathbf{\tau}$ acting as a boundary condition on the leaves. Similarly to
the proof of the first point of Proposition \ref{prop:Sinc}, we have, for any
$y, z \in \mathbbm{T}_l^b$ with $y$ the ancestor of $z$, and every
$\tmmathbf{\eta} \in \Sigma$,
\begin{eqnarray*}
  \left\| \mu^+_l \left( \tmmathbf{\sigma}_z \in \cdummy | \tmmathbf{\sigma}_y
  =\tmmathbf{\eta} \right) - \mu^{\tmmathbf{\tau}}_l \left(
  \tmmathbf{\sigma}_z \in \cdummy | \tmmathbf{\sigma}_y =\tmmathbf{\eta}
  \right) \right\|_{\tmop{TV}} \text{ \ \ } \leqslant \text{ \ \ \ \ \ \ \ \ \
  \ \ \ \ \ \ \ \ \ \ \ \ \ \ } &  & \\
  \Gamma \sum_{x \text{ child of } z} \left\| \mu^+_l \left(
  \tmmathbf{\sigma}_x \in \cdummy | \tmmathbf{\sigma}_y =\tmmathbf{\eta}
  \right) - \mu^-_l \left( \tmmathbf{\sigma}_x \in \cdummy |
  \tmmathbf{\sigma}_y =\tmmathbf{\eta} \right) \right\|_{\tmop{TV}} . &  & 
\end{eqnarray*}
This proves that, as before, every $\rho_{\left( i, j \right)}$ and
$\nu_i^{\pm}$ lies in a given neighborhood $V \subset \mathcal{X}_{1, +}$ of
$\nu_{\infty}$ in the $\left\| . \right\|_{\infty, \mathcal{X}}$ norm, for all
$i$ such that $\min \left( i, l - i \right)$ is large enough depending on $V$
but not on $\tmmathbf{\tau}$. The rest of the argument is identical.

\subsection{Proof of Theorem \ref{thm:mixing}}

Let $\tmmathbf{\tau} \in \Sigma^{\mathbbm{T}_{\infty}^b}$ be such that $b
\gamma \kappa \left( \tmmathbf{\tau} \right) < 1$. We recall that, from
Theorem \ref{thm:decay}, the plus boundary condition always satisfies this
requirement, and that, if $\beta, \tmmathbf{h}, \lambda$ are in the uniqueness
region, we could choose any $\tmmathbf{\tau}$.

\subsubsection{Spectral gap}

Following {\cite{MSW04}} (see Section 3 and the proof of Theorem 4.3), for the
proof of the first part of Theorem \ref{thm:mixing} it is enough to verify
(see Equation (14) in {\cite{MSW04}}) that for all $\varepsilon > 0$, for all
$k \geqslant 1$ large enough depending on $\varepsilon$, for all $l \geqslant
k$, for any $g \in L^2 \left( \Sigma, \mathd \varphi \right)$ and all
$\tmmathbf{\eta} \in \Sigma$,
\begin{eqnarray}
  \tmop{Var}_{\mu_l} \left( \mu_l \left( g \left( \tmmathbf{\sigma}_r \right)
  | \left\{ \tmmathbf{\sigma}_z \right\}_{\left| z \right| = k} \right)
  \right) & \leqslant & \left( \left( 1 + \varepsilon \right) \kappa \gamma b
  \right)^{k/2} \tmop{Var}_{\mu_l} \left( g \left( \tmmathbf{\sigma}_r \right)
  \right)  \label{var:mix}
\end{eqnarray}
where $\mu_l$ is the Gibbs measure on $\mathbbm{T}_l^b$ with boundary
condition $\tmmathbf{\tau}$ on the leaves and extra field $\tmmathbf{\eta}$
acting on the root. In what follows, all the bounds will be uniform in
$\tmmathbf{\eta}$.

Let us denote $K (\tmmathbf{\sigma}_r, .)$ the Radon-Nikodym derivative of the
measure $\mu_l \left( \mu_l \left( \tmmathbf{\sigma}_r' \in \cdummy | \left\{
\tmmathbf{\sigma}_z \right\}_{\left| z \right| = k} \right) |
\tmmathbf{\sigma}_r \right)$ with respect to the marginal distribution of the
root $\nu = \mu_l (\tmmathbf{\sigma}_r \in \cdummy)$. The kernel $K$ is
constructed as follows: given $\tmmathbf{\sigma}_r$ sample the spins at
distance $k$ from the root. Then take these spins and $\tmmathbf{\eta}$ acting
on top of the root as a boundary condition on $\mathbbm{T}_k^b$ and sample
again $\tmmathbf{\sigma}_r'$. It is useful to remark that $K$ is uniformly
bounded. The Cauchy-Schwarz inequality implies that
\begin{eqnarray*}
 \tmop{Var}_{\mu_l}\left( \mu_l \left( g \left( \tmmathbf{\sigma}_r \right)
  \left| \left\{ \tmmathbf{\sigma}_z \right. \right\}_{\left| z \right| = k} \right)
  \right)
&  = & \tmop{Cov}_{\mu_l}\left( \mu_l \left( \left. \mu_l \left( g(\tmmathbf{\sigma}_r) | \left\{
\tmmathbf{\sigma}_z \right\}_{\left| z \right| = k} \right) \right|
\tmmathbf{\sigma}_r \right), g \left( \tmmathbf{\sigma}_r \right) \right) \\
  & \leqslant &  \tmop{Var}_{\mu_l}\left( \mu_l \left( g \left( \tmmathbf{\sigma}_r \right)
  | \left\{ \tmmathbf{\sigma}_z \right\}_{\left| z \right| = k} \right)
  \right)^{1/2}
\tmop{Var}_{\mu_l}\left( g \left( \tmmathbf{\sigma}_r \right) \right)^{1/2}.
\end{eqnarray*}
Then
\begin{gather*}
 \tmop{Var}_{\mu_l}\left( \mu_l \left( g \left( \tmmathbf{\sigma}_r \right)
  | \left\{ \tmmathbf{\sigma}_z \right\}_{\left| z \right| = k} \right)
  \right)  \\
  =  \tmop{Var}_{\mu_l} \left( \int K (\tmmathbf{\sigma}_r,
  \tmmathbf{\sigma}_r') g (\tmmathbf{\sigma}'_r) \mathd \nu
  (\tmmathbf{\sigma}'_r) \right)\\
   =  \frac{1}{2} \nu \otimes \nu \left( \left( \int K
  (\tmmathbf{\sigma}_r^1, \tmmathbf{\sigma}_r') g (\tmmathbf{\sigma}'_r)
  \mathd \nu (\tmmathbf{\sigma}'_r) - \int K (\tmmathbf{\sigma}_r^2,
  \tmmathbf{\sigma}_r') g (\tmmathbf{\sigma}'_r) \mathd \nu
  (\tmmathbf{\sigma}'_r) \right)^2 \right)\\
  \leqslant  \frac{1}{2} \nu \otimes \nu \left( \int \left( K
  (\tmmathbf{\sigma}_r^1, \tmmathbf{\sigma}_r') - K (\tmmathbf{\sigma}_r^2,
  \tmmathbf{\sigma}_r') \right)^2 \mathd \nu (\tmmathbf{\sigma}'_r) \times
  \int \left( g (\tmmathbf{\sigma}'_r) - \mu \left( g \right) \right)^2 \mathd
  \nu (\tmmathbf{\sigma}'_r) \right)\\
   = \frac{1}{2} \tmop{Var}_{\mu_l} (g) \times \nu \otimes \nu \otimes \nu \left[
  \left( K (\tmmathbf{\sigma}_r^1, \tmmathbf{\sigma}_r^3) - K
  (\tmmathbf{\sigma}_r^2, \tmmathbf{\sigma}_r^3) \right)^2 \right]
\end{gather*}
so (\ref{var:mix}) would follow from
\begin{eqnarray*}
  \nu \otimes \nu \otimes \nu \left[ \left( K (\tmmathbf{\sigma}_r^1,
  \tmmathbf{\sigma}_r^3) - K (\tmmathbf{\sigma}_r^2, \tmmathbf{\sigma}_r^3)
  \right)^2 \right] & \leqslant & 2 \left( \left( 1 + \varepsilon \right) \kappa
  \gamma b \right)^k
\end{eqnarray*}
From the boundedness of $K$ we conclude that, for some constant $C$,
\begin{gather*}
  \nu \otimes \nu \otimes \nu \left[ \left( K (\tmmathbf{\sigma}_r^1,
  \tmmathbf{\sigma}_r^3) - K (\tmmathbf{\sigma}_r^2, \tmmathbf{\sigma}_r^3)
  \right)^2 \right]  \\\leqslant  4 C 
\sup_{\tmmathbf{\sigma}_r^1,\tmmathbf{\sigma}_r^2 \in \Sigma}   
  \left\| K (\tmmathbf{\sigma}_r^1, .) \nu
  - K (\tmmathbf{\sigma}_r^2, .) \nu \right\|_{\tmop{TV}}\\
   \leqslant  8 C \sup_{\tmmathbf{\rho} \in \Sigma} \left\| K (+, .) \nu -
  K (\tmmathbf{\rho}, .) \nu \right\|_{\tmop{TV}}.
\end{gather*}
The DLR property, together with the definition of $\Gamma$ at
(\ref{mu:Gamma}), imply that for ordered marginals the TV distance is
comparable with difference of expectation of $\sigma \bullet \tmmathbf{1}$.
Consequently,

\begin{gather*}
[ \left\| K (+, .) \nu - K (\tmmathbf{\rho}, .) \nu \right\|_{\tmop{TV}}
   \\\leqslant  C \left[ \mu_l \left( \mu_l \left( \tmmathbf{\sigma}_r'
  \bullet \tmmathbf{1} | \left\{ \tmmathbf{\sigma}_z \right\}_{\left| z
  \right| = k} \right) | \tmmathbf{\sigma}_r = + \right) - \mu_l \left( \mu_l
  \left( \tmmathbf{\sigma}_r' \bullet \tmmathbf{1} | \left\{
  \tmmathbf{\sigma}_z \right\}_{\left| z \right| = k} \right) |
  \tmmathbf{\sigma}_r =\tmmathbf{\rho} \right) \right]
\end{gather*}
which is clearly maximum if we take $\tmmathbf{\rho}= -$. Now, we argue that
the former difference is bounded by
\begin{eqnarray}
  \left[ \sum_{z : \left| z \right| = k} \mu_l \left( \tmmathbf{\sigma}_z
  \bullet \tmmathbf{1} | \tmmathbf{\sigma}_r = + \right) - \mu_l \left(
  \tmmathbf{\sigma}_z \bullet \tmmathbf{1} | \tmmathbf{\sigma}_r = - \right)
  \right] \times 2 \Gamma \gamma^{k - 1} &  &  \label{eq:diffmag}
\end{eqnarray}
For proving this we need only a slight adaptation of the proof of (ii) in
Claim 4.4 in {\cite{MSW04}}. Consider two spin configurations $\tmmathbf{\xi},
\tmmathbf{\xi}'$ that differ at a single position $z$ with $\left| z \right| =
k$. We can easily construct a coupling of $\mu_l \left( \tmmathbf{\sigma}_r
\in \cdummy | \left\{ \tmmathbf{\sigma}_z =\tmmathbf{\xi}'_z \right\}_{\left|
z \right| = k} \right)$ and $\mu_l \left( \tmmathbf{\sigma}_r \in \cdummy |
\left\{ \tmmathbf{\sigma}_z =\tmmathbf{\xi}_z \right\}_{\left| z \right| = k}
\right)$ for which the two variables differ with probability at most
$\gamma^{k - 1} \Gamma \left\| \tmmathbf{\xi}' -\tmmathbf{\xi} \right\|_1$
according to the definition of $\Gamma$ in Proposition \ref{prop:gamma} and to
that of $\gamma$ at (\ref{gamma}). Now we consider a monotone
coupling of $\left\{ \tmmathbf{\sigma}_z \right\}_{\left| z \right| = k}$
corresponding to the conditions $\tmmathbf{\sigma}_r = \pm$. By applying the
former coupling to an interpolating sequence between these spin configurations
as in proof of (ii) in Claim 4.4 in {\cite{MSW04}}, we conclude the proof of
(\ref{eq:diffmag}).

Finally, $\mu_l \left( \tmmathbf{\sigma}_z \bullet \tmmathbf{1} |
\tmmathbf{\sigma}_r = + \right) - \mu_l \left( \tmmathbf{\sigma}_z \bullet
\tmmathbf{1} | \tmmathbf{\sigma}_r = - \right) \leqslant \left( 1 +
\varepsilon \right)^k \kappa^k$ for all $k \geqslant 0$ large enough according
to the definition of $\kappa$. The proof of the first part of Theorem
\ref{thm:mixing} is complete.

\subsubsection{Mixing time}

We first establish an intermediate step. We fix $l$ and consider the Gibbs
measure on $\mathbbm{T}_l^b$ with boundary condition $\tmmathbf{\tau}$ on the
leaves of $\mathbbm{T}_l^b$ and additional field $\tmmathbf{\eta} \in \Sigma$
at the root. Now we turn to the dynamics. Consider as a starting configuration
the identically plus configuration. Similarly we could consider the
identically minus configuration. Let $P_t^+$ be the distribution of the spin
configuration at time $t$ for the dynamics. We denote by
$h_r^{\tmmathbf{\eta}} \left( t \right)$ the Radon-Nikodym derivative of
$P_t^+$ with respect to $\mu$. Note that $h_r^{\tmmathbf{\eta}} \left( 0
\right) =\tmmathbf{1}_{\left\{ + \right\}} / \mu \left( + \right)$ and
$h_r^{\tmmathbf{\eta}} \left( t \right) = P_t h_r^{\tmmathbf{\eta}} \left( 0
\right)$ because $P_t$ is self-adjoint in $L^2 \left( \mu \right)$.

We also define $T_r = \sup_{\tmmathbf{\eta}} \min \{t : \tmop{Var}
(h_r^{\tmmathbf{\eta}} \left( t \right)) \leqslant 1\}$. Similarly, given $x
\in \mathbbm{T}_l^b \setminus \left\{ r \right\}$, we consider the dynamics
censored everywhere except on the subtree rooted at $x$ and call $h_x^+ \left(
t \right)$ the resulting Radon-Nikodym derivative with respect to the Gibbs
measure conditioned on being plus outside the subtree of $x$. We then define
$T_x = \min \left\{ t : \tmop{Var} \left( h_x^+ \left( t \right) \right)
\leqslant 1 \right\}$.

Now we prove that there exists a constant $t_0$ independent of $l$ such that
\begin{eqnarray}
  T_x & \leqslant & \max_{y \text{ child of } x} T_y + t_0 .  \label{eq:Txt0}
\end{eqnarray}

For this purpose we use censoring (Proposition \ref{prop:censoring}) together
with the first point of Theorem \ref{thm:mixing}. For simplicity we only
consider the case $x = r$. We censor for time $t = \max_{y \text{ child of }
r} T_y$ the root and then run the uncensored dynamics for an extra time $t_0$
to be determined later. Let $\bar{h}_r^{\tmmathbf{\eta}} \left( t \right)$ be
the Radon-Nikodym derivative of the corresponding distribution at time $t$.
Then,
\begin{eqnarray*}
  \tmop{Var} \left( h_r^{\tmmathbf{\eta}} \left( t + t_0 \right) \right) &
  \leqslant & \tmop{Var} \left( \bar{h}_r^{\tmmathbf{\eta}} \left( t + t_0
  \right) \right)\\
  & = & \tmop{Var} \left( P_{t_0}  \bar{h}_r^{\tmmathbf{\eta}} \left( t
  \right) \right)\\
  & \leqslant & e^{- 2 \tmop{gap} \times t_0 } \tmop{Var} \left(
  \bar{h}_t^{\tmmathbf{\eta}} \left( t \right) \right) .
\end{eqnarray*}
Now we prove that $\tmop{Var} \left( \bar{h}_r^{\tmmathbf{\eta}} \left( t
\right) \right)$ is bounded uniformly in $l$. By construction we have
\begin{eqnarray*}
  \bar{h}_r^{\tmmathbf{\eta}} \left( t \right) \left( \tmmathbf{\sigma}
  \right) & = & \frac{\tmmathbf{1}_{\left\{ + \right\}} \left(
  \tmmathbf{\sigma}_r \right)}{\mu \left( \tmmathbf{\sigma}_r = + \right)}
  \prod_{y \text{ child of } r} h_y^+ \left( t \right) \left(
  \tmmathbf{\sigma}_{\mathbbm{T}_y} \right)
\end{eqnarray*}
where $\tmmathbf{\sigma}_{\mathbbm{T}_y}$ is the restriction of
$\tmmathbf{\sigma}$ to the subtree rooted at $y$. Finally, we remark that
\begin{eqnarray*}
  \tmop{Var} \left( \bar{h}_r^{\tmmathbf{\eta}} \left( t \right) \right) &
  \leqslant & \mu \left( \left( \bar{h}_r^{\tmmathbf{\eta}} \left( t \right)
  \right)^2 \right)\\
  & = & \mu (\tmmathbf{\sigma}_r = +) \mu \left( \left(
  \bar{h}_r^{\tmmathbf{\eta}} \left( t \right) \right)^2 | \tmmathbf{\sigma}_r
  = + \right)\\
  & = & \frac{1}{\mu (\tmmathbf{\sigma}_r = +)} \prod_{y \text{ child of } r}
  \mu \left( \left( h_y^+ \left( t \right) \left(
  \tmmathbf{\sigma}_{\mathbbm{T}_y} \right) \right)^2 | \tmmathbf{\sigma}_r =
  + \right)\\
  & = & \frac{1}{\mu (\tmmathbf{\sigma}_r = +)} \prod_{y \text{ child of } r}
  \left[ \tmop{Var}_{\mu \left( \cdummy | \tmmathbf{\sigma}_r = + \right)}
  \left( h_y^+ \left( t \right) \left( \tmmathbf{\sigma}_{\mathbbm{T}_y}
  \right) \right) + 1 \right]
\end{eqnarray*}
which is smaller than $2^b / \mu \left( \tmmathbf{\sigma}_r = + \right)$
according to the definition of $t$. If we take $t_0$ such that
\begin{eqnarray*}
  ^{} \frac{2^b \exp \left( - 2 \tmop{gap} \times t_0 \right)}{\mu \left(
  \tmmathbf{\sigma}_r = + \right)} & \leqslant & 1
\end{eqnarray*}
then $t_0$ is bounded uniformly in $l$ and we are done. In conclusion, we have
shown that $T_r \leqslant lt_0$.

\begin{remark}
  We observe that the recursive inequality (\ref{eq:Txt0}) is analogous to the
  one obtained in {\cite{MSW04}} (see Lemma 5.8 there) for the logarithmic
  Sobolev constant. In our context, this constant is easily seen to be
  infinite because of configurations with an arbitrary large number of flips.
\end{remark}

We know have to consider the dynamics starting from an arbitrary spin
configuration $\tmmathbf{\xi}$. Choose $t = T_r + cl$ for some $c > 0$ to be
chosen later on. We have of course
\begin{eqnarray*}
  \|P^{\tmmathbf{\xi}}_t - \mu \|_{\tmop{TV}} & \leqslant &
  \|P^{\tmmathbf{\xi}}_t - P^+_t \|_{\tmop{TV}} + \|P^+_t - \mu \|_{\tmop{TV}}
  .
\end{eqnarray*}
On one hand,
\begin{eqnarray}
  \|P^+_t - \mu \|_{\tmop{TV}} & = & \left\| h_r^{\tmmathbf{\eta}} \left( t
  \right) - 1 \right\|_{L^1 (\mu)} \leqslant \left\| h_r^{\tmmathbf{\eta}}
  \left( t \right) - 1 \right\|_{L^2 (\mu)} \leqslant e^{- \tmop{gap} (t -
  T_r)} \leqslant e^{- \tmop{gap} \times cl} \label{TVp}
\end{eqnarray}
according to the definition of $T_r$.

For the remaining part $\|P^{\tmmathbf{\xi}}_t - P^+_t \|_{\tmop{TV}}$ we use
a coupling argument. Call $s = T_r + cl / 2$. We consider $\Psi_s$ a monotone
coupling of $P^{\tmmathbf{\xi}}_s \prec P^+_s$ and denote $\left(
\tmmathbf{\sigma}, \tmmathbf{\sigma}^+ \right)$ its variables. According to
Markov's inequality, we have
\begin{eqnarray*}
  \Psi_s (\exists x, \tmmathbf{\sigma}_x \bullet 1 \leqslant
  \tmmathbf{\sigma}_x^+ \bullet 1 - \varepsilon) & \leqslant & \sum_x \Psi_s
  (\tmmathbf{\sigma}_x \bullet 1 \leqslant \tmmathbf{\sigma}_x^+ \bullet 1 -
  \varepsilon)\\
  & \leqslant & \varepsilon^{- 1} \sum_x P^+_s (\tmmathbf{\sigma}_x \bullet
  1) - P^{\tmmathbf{\xi}}_s (\tmmathbf{\sigma}_x \bullet 1)\\
  & \leqslant & \varepsilon^{- 1} \sum_x P^+_s (\tmmathbf{\sigma}_x \bullet
  1) - P^-_s (\tmmathbf{\sigma}_x \bullet 1)\\
  & \leqslant & \frac{2 b^l \beta}{\varepsilon} (\|P^+_s - \mu \|_{\tmop{TV}}
  +\|P^-_s - \mu \|_{\tmop{TV}}) .
\end{eqnarray*}
Now we take $\varepsilon = b^{- 2 l}$. According to the last display and to (\ref{TVp}) the probability $\Psi_s (\exists x, \tmmathbf{\sigma}_x \bullet 1\leqslant
  \tmmathbf{\sigma}_x^+ \bullet 1 - \varepsilon)$ can be made arbitrary small by taking $c$ large. So with high probability under $\Psi_s$, at not position the spin configurations differ by more than $\varepsilon$ in $L^1 \left( \left[ 0, \beta \right] \right)$ distance at time $s= T_r + cl / 2 = t-cl / 2$.   In the remaining time $cl / 2$, we
update the spin configurations at the same positions according to an optimal
coupling of the marginals. Using Proposition \ref{prop:gamma} we see that, as
long as the spin configurations differ at every $x$ by at most $\varepsilon$
in $L^1 \left( \left[ 0, \beta \right] \right)$ distance, each update put
locally the same spin with probability at least $1 - \Gamma \varepsilon$.
Therefore we conclude that $\|P^{\tmmathbf{\xi}}_t - P^+_t \|_{\tmop{TV}}$ can
be made arbitrarily small if $c$ is large enough. This concludes the proof
that the mixing time is bounded by $Cl$.

{\bf Acknowledgements.} We wish to thank Guilhem Semerjian and Francesco Zamponi for very stimulating discussions about their paper. One of us (M.~Wouts) thanks Jean-Fran\c{c}ois Quint for useful discussion about positive operators, and acknowledges the generous hospitality by University of Rome 3.


\begin{thebibliography}{1}

\bibitem{CDP}
Alessandra Cipriani and Paolo Dai~Pra.
\newblock Decay of correlations for quantum spin systems with a transverse
  field: A dynamic approach.
\newblock {\em arXiv:1005.3547}, 2010.

\bibitem{Io09}
Dmitry Ioffe.
\newblock Stochastic geometry of classical and quantum {I}sing models.
\newblock In {\em Methods of contemporary mathematical statistical physics},
  volume 1970 of {\em Lecture Notes in Math.}, pages 87--127. Springer, Berlin,
  2009.

\bibitem{KR50}
M.~G. Kre{\u\i}n and M.~A. Rutman.
\newblock Linear operators leaving invariant a cone in a {B}anach space.
\newblock {\em Amer. Math. Soc. Translation}, 1950(26):128, 1950.

\bibitem{KRSZ08}
F.~{Krzakala}, A.~{Rosso}, G.~{Semerjian}, and F.~{Zamponi}.
\newblock {Path-integral representation for quantum spin models: Application to
  the quantum cavity method and Monte Carlo simulations}.
\newblock {\em Phys. Rev. B}, 78(13):134428, 2008.

\bibitem{LPW09}
David~A. Levin, Yuval Peres, and Elizabeth~L. Wilmer.
\newblock {\em Markov chains and mixing times}.
\newblock American Mathematical Society, Providence, RI, 2009.
\newblock With a chapter by James G. Propp and David B. Wilson.

\bibitem{Ma97}
Fabio Martinelli.
\newblock Lectures on {G}lauber dynamics for discrete spin models.
\newblock In {\em Lectures on probability theory and statistics
  ({S}aint-{F}lour, 1997)}, volume 1717 of {\em Lecture Notes in Math.}, pages
  93--191. Springer, Berlin, 1999.

\bibitem{MSW04}
Fabio Martinelli, Alistair Sinclair, and Dror Weitz.
\newblock Glauber dynamics on trees: boundary conditions and mixing time.
\newblock {\em Comm. Math. Phys.}, 250(2):301--334, 2004.

\bibitem{PeresUBC}
Yuval Peres.
\newblock Mixing for {M}arkov chains and spin systems.
\newblock {\em Lectures at UBC}, 2005.

\end{thebibliography}
\end{document}